\documentclass[12pt]{amsart}
\usepackage{fullpage}

\usepackage{hyperref}

\usepackage{xcolor}

\usepackage{xypic}

\usepackage{enumerate}
\usepackage{amsthm}
\usepackage{amsmath}
\usepackage{amssymb}
\usepackage{amsfonts}

\newtheorem{thm}{Theorem}[section]
\newtheorem*{thm*}{Theorem} 
\newtheorem{lem}[thm]{Lemma}
\newtheorem{Def}[thm]{Definition}

\newtheorem{prop}[thm]{Proposition}
\newtheorem{cor}[thm]{Corollary}
\newtheorem{rem}[thm]{Remark}

\theoremstyle{definition}{\newtheorem{ex}[thm]{Example}}

\newcommand{\N}{\ensuremath{\mathbb{N}}}
\newcommand{\C}{\ensuremath{\mathbb{C}}}
\newcommand{\Z}{\ensuremath{\mathbb{Z}}}
\newcommand{\Q}{\ensuremath{\mathbb{Q}}}
\newcommand{\R}{\ensuremath{\mathbb{R}}}

\newcommand{\s}{\ensuremath{\sigma}}
\newcommand{\f}{\phi}
\newcommand{\de}{\delta}
\newcommand{\ds}{{\delta\sigma}}

\newcommand{\G}{\ensuremath{\mathcal{G}}}

\newcommand{\X}{\ensuremath{\mathcal{X}}}
\newcommand{\Y}{\ensuremath{\mathcal{Y}}}

\newcommand{\GL}{\operatorname{GL}}

\newcommand{\Gm}{\mathbb{G}_m}
\newcommand{\Ga}{\mathbb{G}_a}

\newcommand{\Hom}{\operatorname{Hom}}
\newcommand{\Aut}{\operatorname{Aut}}
\newcommand{\Gal}{\underline{\operatorname{Gal}}^{\ds}}

\newcommand{\Spec}{\operatorname{Spec}}

\newcommand{\Frac}{\ensuremath{\mathrm{Frac}}}

\def\H{\mathcal{H}}

\newcommand{\ks}{$k$-$\s$}

\newcommand{\ida}{\mathfrak{a}}
\newcommand{\V}{\mathbb{V}}
\newcommand{\hs}{{{}^\sigma\!}}
\newcommand{\hsi}{^{\sigma^i}\!}
\newcommand{\pis}{\pi_0^\sigma}
\newcommand{\ssetale}{{strongly $\sigma$-\'{e}tale}}

\title{Algebraic groups as difference Galois groups of linear differential equations}
\author{Annette Bachmayr and Michael Wibmer}
\address{Michael Wibmer, Institute of Analysis and Number Theory, Graz University of Technology, Kopernikusgasse 24, 8010 Graz, Austria}
\email{wibmer@math.tugraz.at}

\address{Annette Bachmayr (n\'{e}e Maier), Institute of Mathematics, Univsersity of Mainz, Staudingerweg 9, 55128 Mainz, Germany}
\email{abachmay@uni-mainz.de}

\thanks{The first author was funded by the Deutsche Forschungsgemeinschaft (DFG) - grants MA6868/1-1, MA6868/1-2 and by the Alexander von Humboldt foundation through a Feodor Lynen fellowship. The second author was supported by the NSF grants DMS-1760212, DMS-1760413, DMS-1760448 and the Lise Meitner grant \mbox{M-2582-N32} of the Austrian Science Fund FWF}



\date{\today}

\begin{document}

	\begin{abstract}
		We study the inverse problem in the difference Galois theory of linear differential equations over the difference-differential field $\C(x)$ with derivation $\frac{d}{dx}$ and endomorphism $f(x)\mapsto f(x+1)$. Our main result is that every linear algebraic group, considered as a difference algebraic group, occurs as the difference Galois group of some linear differential equation over $\C(x)$. 
	\end{abstract}
	
\maketitle

\section*{Introduction}
The Galois group of a polynomial over a field is a finite group. The inverse problem in the Galois theory of polynomials asks to determine, for a given field, which finite groups occur. For example, for the field $\C(x)$ of rational functions over $\C$, it is known that every finite group occurs (see e.g., \cite[Cor. 3.4.4]{Samuely:GaloisGroupsAndFundamentalGroups}).

The Galois group of a linear differential equation over a differential field is a linear algebraic group. The inverse problem in the Galois theory of linear differential equations asks to determine, for a given differential field, which linear algebraic groups occur. For example, for the field $\C(x)$ with derivation $\frac{d}{dx}$, it is known that every linear algebraic group occurs. This was first proved in \cite{Tretkoff}, based on the solution of Hilbert's 21\textsuperscript{st} problem.

A difference-differential field is a field equipped with two commuting operators, a derivation and an endomorphism, usually denoted with $\s$. The $\s$-Galois group of a linear differential equation over a difference-differential field is a linear difference algebraic group, i.e., a subgroup of a general linear group defined by algebraic difference equations in the matrix entries. The problem we are concerned with in this article is the inverse problem in the $\s$-Galois theory of linear differential equations. It asks to determine, for a given difference-differential field, which difference algebraic groups occur.
We are mainly interested in the difference-differential field $\C(x)$ with derivation $\frac{d}{dx}$ and endomorphism $\s\colon \C(x)\to\C(x),\ f(x)\mapsto f(x+1)$.

As we will show, not every difference algebraic group occurs as a $\s$-Galois group of a linear differential equation over $\C(x)$.  In fact, constant subgroups of unipotent linear algebraic groups do not occur (Corollary \ref{cor: unipotent}) and moreover, we isolate two properties that any $\s$-Galois group over $\C(x)$ must have (Theorem \ref{theo: sgalois groups are sreduced and sconnected}). On the positive side, our main result is the following: 

\begin{thm*}[Theorem \ref{theo: main}]
	Every linear algebraic group, considered as a difference algebraic group, occurs as a $\s$-Galois group over $\C(x)$.
\end{thm*}
If a linear algebraic group $G$, considered as a difference algebraic group, occurs as a $\s$-Galois group of a linear differential equation $y'=Ay$, then the Galois group of the linear differential equation $y'=Ay$ is the linear algebraic group $G$. Thus the above theorem generalizes the solution of the inverse problem in the Galois theory of linear differential equations over $\C(x)$.

Our main tool for the proof of the above theorem is patching. In fact, we establish a general patching result (Theorem \ref{thm: patching}) for $\s$-Picard-Vessiot rings over difference-differential fields that we deem of independent interest. Here a $\s$-Picard-Vessiot ring is the analog of the splitting field in the Galois theory of polynomials. This patching result is analogous to known patching results in the Galois theories of linear differential equations (\cite[Theorem 2.4]{BachmayrHarbaterHartmann:DifferentialGaloisGroupsOverLaurentSeriesFields}) and parameterized linear differential equations (\cite[Theorem 2.2]{param_LAG}) that turned out to be very useful in the study of the corresponding inverse problems. We therefore expect Theorem \ref{thm: patching} to have further applications in the study of the inverse problem in the $\s$-Galois theory of linear differential equations.
  
\medskip

To put our results into perspective, let us review the state of the art of the inverse problem in the various Galois theories. The three most relevant Galois theories for us are the following: 
\begin{enumerate}
	\item The Galois theory of linear differential equations, where the Galois groups are linear algebraic groups. See e.g. \cite{SingerPut:differential}. 
	\item The Galois theory of parameterized linear differential equations, where the Galois groups are differential algebraic groups. See \cite{CassSin} and \cite{Landesman:GeneralizedDifferentialGaloisTheory}.
	\item The $\s$-Galois theory of linear differential equations, where the Galois groups are difference algebraic groups. See \cite{DiVizioHardouinWibmer:DifferenceGaloisTheoryOfLinearDifferentialEquations}.
\end{enumerate}  

\subsection*{Galois theory of linear differential equations}

Building on work of several authors, the inverse problem in the Galois theory of linear differential equations over $k(x)$, where $k$ is an algebraically closed field of characteristic zero, was eventually solved in \cite{Hartmann:OnTheInverseProblemInDifferentialGaloisTheory}: All linear algebraic groups over $k$ occur as Galois groups. For non-algebraically closed fields $k$, there are only partial results. For example, if $k$ is a Laurent series field, all linear algebraic groups over $k$ occur as Galois groups (\cite[Theorem 4.14]{BachmayrHarbaterHartmann:DifferentialGaloisGroupsOverLaurentSeriesFields}) and the same holds for $k=\Q_p$ (\cite{BachmayrHartmannHarbaterPopLarge}). 

Going beyond the solution of the inverse problem, there has been recent progress in the study of differential embedding problems and the structure of the absolute differential Galois group of $k(x)$ using patching techniques. See \cite{BachmayrHarbaterHartmannWibmer:DifferentialEmbeddingProblemsOverComplexFunctionFields,BachmayrHarbaterHartmann:DifferentialEmbeddingProblemsOverLaurantSeriesFields,BachmayrHartmannHarbaterPopLarge,Wibmer:FreeProalgebraicGroups,BachmayrHarbaterHartmannWibmer:FreeDifferentialGaloisGroups}.
%

\subsection*{Galois theory of parameterized linear differential equations}
If the coefficients of a linear differential equation depend on an auxiliary parameter, one can differentiate the solutions with respect to this parameter. The Galois group of a parameterized linear differential equation is a differential algebraic group that measures the algebraic relations among the solutions and their derivatives with respect to the auxiliary parameter. From an algebraic perspective, this setup is modeled by considering a field $K$ with two commuting derivations $\de$ and $\partial$. A typical example is $K=\C(t,x)$, where $x$ is the main variable and $t$ the parameter, i.e., we are interested in linear differential equations with respect to $\de=\frac{d}{dx}$ and the $\partial=\frac{d}{dt}$-derivatives of their solutions.
More generally, if $k$ is a field equipped with a derivation $\partial$, then $K=k(x)$ is naturally equipped with two commuting derivations, $\de=\frac{d}{dx}$ and $\partial$, where $\partial(x)=0$.
 This Galois theory and its variants have proven to be very useful in questions of hypertranscendence. See \cite{HardouinSinger:DifferentialGaloisTheoryofLinearDifferenceEquations}, \cite{Arreche:AGaloisTheoreticProofOfTheDifferentialTranscendenceOfTheIncompleteGammaFunction}, \cite{HardouinMinchenkoOvchinnikov:CalculatingDifferentialGaloisGroupsOfParameterizedDifferentialEqautionsWithApplications}, \cite{DreyfusHardouinRoques:HypertranscendenceOfSolutionsOfMahlerFunctions}, \cite{DiVizio:ApprocheGaloisienneDeLaTranscendanceDifferentielle}.
 
The inverse problem in this Galois theory is not well-understood. For example, it is not known which differential algebraic groups occur as Galois groups over $\C(t,x)$. The most comprehensive result is only available under strong assumptions on $k$: Building on \cite{MR2975151}, it was shown in \cite{Dreyfus} that if $k$ is a universal $\partial$-field, then a differential algebraic group over $k$ is a Galois group over $k(x)$ if and only if it is the Kolchin closure of a finitely generated subgroup. For certain differential algebraic groups, including linear algebraic groups, the latter condition was translated into group theoretic conditions in \cite{Singer:LinearAlgebraicGroupsAsParameterizedPicardVessiotGaloisGroups} and \cite{MinchenkoOvchinnikovSinger:UnipotentDifferentialAlgebraicGroupsAsParameterizedDifferentialGaloisGroups}.

For $k$ a Laurent series field, it was shown in \cite{param_LAG} that a large class of linear algebraic groups, considered as differential algebraic groups, occur as Galois groups. Also certain differential algebraic groups occur in this situation (\cite{MR3774417}). On the other hand, it is also shown in \cite{MR3774417} that many differential algebraic subgroups of the additive or the multiplicative group do not occur over $k(x)$, unless the $\partial$-field $k$ is fairly big.

\subsection*{$\s$-Galois theory of linear differential equations}
The $\s$-Galois theory of linear differential equations is similar to the Galois theory of parameterized linear differential equations. Again one considers a linear differential equation depending on a parameter $t$. But instead of deriving the solutions with respect to $t$, one applies a discrete transformation to $t$ and the solutions, e.g., $t\mapsto t+1$. The Galois groups are difference algebraic groups and they measure the algebraic relations among the solutions and their transforms under a discrete transformation usually denoted with $\s$. From an algebraic perspective, this setup is modeled by considering a field $K$, with two commuting operators, a derivation $\de$ and an endomorphism $\s$. The inverse problem in this Galois theory is wide open. It appears that beyond some initial observations in \cite{DiVizioHardouinWibmer:DifferenceAlgebraicRel} nothing is known. In this paper, for the first time, a significant class of difference algebraic groups is shown to occur as Galois groups.

It should also be noted that in \cite{ArrecheSinger}, the authors consider $\delta$-Galois groups of linear \emph{difference} equations over $\C(x)$ with respect to the difference operator either the shift $x\mapsto x+1$, a $q$-dilation $x\mapsto qx$ or a Mahler operator $x\mapsto x^q$. For certain classes of solvable differential algebraic groups, they characterize which groups occur as $\delta$-Galois groups. It seems plausible that using similar methods as in \cite{ArrecheSinger} it could also be shown that certain solvable difference algebraic groups (that are not algebraic groups) occur as $\sigma$-Galois groups.

\medskip

The direct problem in the above Galois theories is to compute the Galois group of a given (parameterized) linear differential equation. We note that progress in the inverse problem can be helpful for the direct problem. For example, if it is already known that the Galois group of a given (parameterized) differential equation is non-trivial and contained in a certain group $G$, the information, that no non-trivial subgroup of $G$ is a Galois group would already imply that the searched for Galois group equals $G$. As we show (Prop. \ref{prop do not occur}) this occurs, e.g. for the additive group $G=\Ga$.

\medskip

We conclude the introduction with an overview of the article. In the first section we recall and establish some basic preparatory definitions and results concerning difference algebraic groups and the $\s$-Galois theory of linear differential equations. The main patching result (Theorem \ref{thm: patching}) is established in Section 3. Roughly speaking, it states that if two difference algebraic subgroups $H_1$ and $H_2$ of some difference algebraic group occur as Galois groups in a certain compatible fashion, then also the difference algebraic group generated by $H_1$ and $H_2$ occurs. It follows from this result that to realize all linear algebraic groups (considered as difference algebraic groups) as Galois groups it suffices to realize certain building blocks, namely, the multiplicative group, the additive group and finite cyclic groups. Building on work in the second section, these building blocks are then dealt with in Section 4. Our main result, that all linear algebraic groups (considered as difference algebraic groups) occur as Galois groups over $\C(x)$, where $\s(f(x))=f(x+1)$, is then established in Section 5. In the final section we show that not all difference algebraic groups occur as Galois groups over $\C(x)$. In fact, we isolate two properties that any Galois group over $\C(x)$ must have. The fist property is that is has to be $\s$-reduced. This follows rather directly from the fact that $\s\colon \C(x)\to \C(x)$ is bijective. The second property is $\s$-connectedness. This boils down to the fact that $\C(x)$ does not have any finite difference field extensions. We also show that no proper non-trivial difference algebraic subgroup of the additive group $\Ga$ is a $\s$-Galois group over $\C(x)$ and deduce from this that also the constant points of unipotent linear algebraic groups do not occur.

\medskip The authors are grateful to Thomas Dreyfus and David Harbater for helpful discussions related to the content of this paper.

\section{Basics on difference Galois theory}

In this section we recall the necessary definitions concerning difference algebraic groups and $\s$-Picard-Vessiot theory. We also establish some results of a preparatory nature.

All rings are assume to be commutative and unital.

\subsection{Difference algebraic groups}

We begin by recalling some basic notions from difference algebra. Standard references for difference algebra are \cite{Cohn:difference} and \cite{Levin:difference}. 
For more background on difference algebraic groups see \cite[Appendix A]{DiVizioHardouinWibmer:DifferenceGaloisTheoryOfLinearDifferentialEquations} or \cite{Wibmer:FinitenessProperties}.

A \emph{difference ring}, or \emph{$\s$-ring} for short, is a ring $R$ together with an endomorphism $\s\colon R\to R$. A morphism $\psi\colon R\to S$ of $\s$-rings is a morphism of rings such that
$$
\xymatrix{
R \ar^\psi[r] \ar_\s[d] & S \ar^\s[d] \\
R \ar^\psi[r] & S	
}
$$
commutes. A $\s$-ring is a \emph{$\s$-field} if the underlying ring is a field.

Let $k$ be a $\s$-ring. A \emph{\ks-algebra} is a $\s$-ring $R$ together with a morphism $k\to R$ of $\s$-rings. A morphism of \ks-algebras is a morphism of $k$-algebras that is also a morphism of $\s$-rings.
For a subset $B$ of $R$, the smallest \ks-subalgebra of $R$ that contains $B$ is denoted by $k\{B\}$. Note that $k\{B\}$ is generated by $B,\s(B),\s^2(B),\ldots$ as a $k$-algebra. If $R=k\{B\}$ for a finite subset $B$ of $R$ then $R$ is called \emph{finitely $\s$-generated}. If $R$ and $S$ are \ks-algebras, then $R\otimes_k S$ is a \ks-algebra with $\s$ defined by $\s(r\otimes s)=\s(r)\otimes \s(s)$ for $r\in R$ and $s\in S$.

An ideal $\ida$ of a $\s$-ring $R$ is a \emph{$\s$-ideal} if $\s(\ida)\subseteq \ida$. In this case $R/\ida$ is naturally a $\s$-ring.

The \emph{$\s$-polynomial ring} $k\{y\}=k\{y_1,\ldots,y_n\}$ over $k$ in the $\s$-variables $y_1,\ldots,y_n$ is the polynomial ring over $k$ in the variables $\s^i(y_j)$ ($1\leq j\leq n$, $0\leq i$) with $\s\colon k\{y\}\to k\{y\}$ extending $\s\colon k\to k$ as suggested by the naming of the variables. If $f\in k\{y_1,\ldots,y_n\}$ is a $\s$-polynomial and $x=(x_1,\ldots,x_n)\in R^n$ for some \ks-algebra $R$, then the element $f(x)\in R$ is obtained from $f$ by substituting $\s^i(y_j)$ with $\s^i(x_j)$.
For a \ks-algebra $R$ and $F\subseteq k\{y_1,\ldots,y_n\}$ we set
$\V_R(F)=\{x\in R^n\ |\  \text{for all }f\in F: f(x)=0 \}.$ Note that $R\rightsquigarrow \V_R(F)$ is naturally a functor from the category of \ks-algebras to the category of sets.

Let $k$ be a $\s$-field. A \emph{$\s$-variety $X$ over $k$} is a functor $R\rightsquigarrow X(R)$ from the category of \ks-algebras to the category of sets that is isomorphic (as a functor) to a functor of the form $R\rightsquigarrow\V_R(F)$ for some $n\geq 1$ and $F\subseteq k\{y_1,\ldots,y_n\}$. Thus a functor $X$ from the category of \ks-algebras to the category of sets is a $\s$-variety if and only if it is representable by a finitely $\s$-generated \ks-algebra, i.e., there exists a finitely $\s$-generated \ks-algebra $S$ such that $X\simeq \Hom(S,-)$. By the Yoneda Lemma, the \ks-algebra $S$ is uniquely determined (up to an isomorphism) by $X$. We therefore denote it with $k\{X\}$ and call it the \emph{coordinate ring} of $X$.
A morphism $\f\colon X\to Y$ of $\s$-varieties is a morphism of functors (i.e., a natural transformation). Again, by the Yoneda Lemma, the category of $\s$-varieties over $k$ is anti-equivalent to the category of finitely $\s$-generated \ks-algebras. The morphism dual to $\f\colon X\to Y$ is denoted by $\f^*\colon k\{Y\}\to k\{X\}$.

A \emph{$\s$-closed $\s$-subvariety} $Y$ of a $\s$-variety $X$ is a subfunctor $Y$ of $X$ defined by a $\s$-ideal $\ida$ of $k\{X\}$, i.e., $Y(R)=\{\psi\in\Hom(k\{X\},R)\ | \ \ida\subseteq\ker(\psi)\}\subseteq\Hom(k\{X\},R)=X(R)$ for every \ks-algebra $R$. Then $Y$ is a $\s$-variety with $k\{Y\}=k\{X\}/\ida$. In terms of equations, if $F\subseteq G\subseteq k\{y\}$, then $R\rightsquigarrow \V_R(G)$ is a $\s$-closed $\s$-subvariety of $R\rightsquigarrow \V_R(F)$. 

If $\f\colon X\to Y$ is a morphism of $\s$-varieties, there exists a unique $\s$-closed $\s$-subvariety $\f(X)$ of $Y$ such that $\f$ factors through $\f(X)$ and for every other $\s$-closed $\s$-subvariety $Z$ of $Y$ such that $\f$ factors through $Z$ we have $\f(X)\subseteq Z$ (\cite[Lemma 1.5]{Wibmer:FinitenessProperties}). Indeed, $\f(X)$ is the $\s$-closed $\s$-subvariety of $Y$ defined by the kernel of $\f^*\colon k[Y]\to k[X]$.
A morphism $\f\colon X\to Y$ of $\s$-varieties is a \emph{$\s$-closed embedding} if it induces an isomorphism between $X$ and a $\s$-closed $\s$-subvariety of $Y$, i.e., the morphism $X\to \f(X)$ is an isomorphism. 
This is equivalent to $\f^*\colon k\{Y\}\to k\{X\}$ being surjective.

The category of $\s$-varieties over $k$ has products. Indeed, if $X$ and $Y$ are $\s$-varieties, the functor $R\rightsquigarrow X(R)\times Y(R)$ is a product of $X$ and $Y$ with coordinate ring $k\{X\times Y\}=k\{X\}\otimes_k k\{Y\}$.

If $X$ is a $\s$-variety over $k$ and $k'/k$ is an extension of $\s$-fields, then $X_{k'}$ denotes the $\s$-variety over $k'$ obtained from $X$ by base change from $k$ to $k'$, i.e., $X_{k'}(R')=X(R')$ for every $k'$-$\s$-algebra $R'$ and $k'\{X_{k'}\}=k\{X\}\otimes_k k'$. 

For every $k$-algebra $R$, there exists a \ks-algebra $[\s]_k R$ together with a morphism $R\to [\s]_k R$ of $k$-algebras such that for every \ks-algebra $S$ and $k$-algebra morphism $R\to S$ there exists a unique morphism $[\s]_kR\to S$ of \ks-algebra such that
$$
\xymatrix{
R \ar[rr] \ar[rd]& & [\s]_k R \ar[ld] \\
& S 	
}
$$ 
commutes. Explicitly, $[\s]_kR$ can be described as follows: For $i\geq 0$ let ${\hsi R}=R\otimes_k k$ denote the $k$-algebra obtained from the $k$-algebra $R$ by base change via $\s^i\colon k\to k$. Set $R[i]=R\otimes _k{\hs R}\otimes_k\ldots\otimes_k{\hsi R}$ and let $[\s]_kR$ denote the union of the $R[i]$'s. We turn $[\s]_kR$ into a \ks-algebra by setting
$$\s((r_0\otimes \lambda_0)\otimes\ldots\otimes (r_i\otimes\lambda_i))=(1\otimes 1)\otimes(r_0\otimes\s(\lambda_0))\otimes\ldots\otimes (r_i\otimes\s(\lambda_i))\in R[i+1]$$ for
$(r_0\otimes \lambda_0)\otimes\ldots\otimes(r_i\otimes\lambda_i)\in R[i]$.
%
%
If $\psi\colon R\to S$ is a morphism of $k$-algebras, then so is $R\to S\to [\s]_k S$ and we obtain a morphism $[\s]_k\psi \colon [\s]_k R\to [\s]_k S$ of \ks-algebras. For later use, we record a lemma:

\begin{lem} \label{lemma: sigmaization is injective}
Let $k$ be a $\s$-field and let $\psi\colon R\to S$ be an injective morphism of $k$-algebras. Then also $[\s]_k\psi\colon [\s]_kR\to [\s]_k S$ is injective.
\end{lem}
\begin{proof}
	 The restriction of $[\s]_k\psi$ to $R[i]$ is $\psi\otimes {\hs \psi}\otimes\ldots\otimes {\hsi \psi}$ which is injective. Here, ${\hsi \psi}\colon {\hsi R}\to {\hsi S}$ is obtained from $\psi \colon R\to S$ by base change via $\s^i\colon k\to k$.
\end{proof}	
Any affine scheme $\X$ of finite type over $k$ can be interpreted as a $\s$-variety. Indeed, the functor $[\s]_k\X$  from the category of \ks-algebras to the category of sets defined by $R\rightsquigarrow \X(R)$ is a $\s$-variety. If $\X=\Spec(k[\X])$, then $k\{[\s]_k\X\}=[\s]_k k[\X]$.
If $\f\colon \X\to \Y$ is a morphism of affine schemes of finite type over $k$, then $[\s]_k\f\colon [\s]_k\X\to [\s]_k\Y$ defined by $([\s]_k\f)_R=\f_R\colon\X(R)\to\Y(R)$ for any \ks-algebra $R$ is a morphism of $\s$-varieties.
A $\s$-closed $\s$-subvariety $Y$ of $\X$ is a $\s$-closed $\s$-subvariety of $[\s]_k\X$. Such a $Y$ is defined by a $\s$-ideal $\ida$ of $[\s]_kk[\X]=\cup_{i\geq 0}k[\X][i]$. The closed subscheme $Y[i]$ of $\X\times{\hs\X}\times\ldots\times{\hsi\X}$ defined by $\ida\cap k[\X][i]$ is called the \emph{$i$-th order Zariski closure} of $Y$ in $\X$. Here ${\hsi \X}$ denotes the affine scheme obtained from $\X$ by base change via $\s^i\colon k\to k$.
Note that the $i$-th order Zariski closure of $[\s]_k\X$, considered as a $\s$-closed $\s$-subvariety of $\X$, is $\X\times{\hs\X}\times\ldots\times{\hsi\X}$.

A \emph{$\s$-algebraic group} $G$ (over $k$) is a group object in the category of $\s$-varieties over $k$. For example, if $\G$ is an affine group scheme of finite over $k$, then $[\s]_k\G$ is a $\s$-algebraic group over $k$. A \emph{$\s$-closed subgroup} $H$ of a $\s$-algebraic group $G$ is a $\s$-subvariety such that $H(R)$ is a subgroup of $G(R)$ for every \ks-algebra $R$. A $\s$-closed subgroup of an affine group scheme $\G$ of finite type over $k$ is a $\s$-closed subgroup of $[\s]_k\G$. If $\H$ is a closed subgroup of $\G$, then $[\s]_k\H$ is a $\s$-closed subgroup of $[\s]_k\G$.
%

\begin{lem}[{\cite[Prop. 2.16]{Wibmer:FinitenessProperties}}] \label{lemma: linearization}
	Every $\s$-algebraic group is isomorphic to a $\s$-closed subgroup of $\GL_n$ for some $n\geq 1$. 
\end{lem}

Let $H_i$, $i\in I$ be a family of $\s$-closed subgroups of a $\s$-algebraic group $G$. Since the intersection of a family of $\s$-closed subgroups of $G$ is a $\s$-closed subgroup of $G$, we see that there exists a smallest $\s$-closed subgroup $\langle H_i \ | \ i\in I\rangle$ such that $H_i$ is contained in $\langle H_i \ | \ i\in I\rangle$ for every $i\in I$. If $G=\langle H_i \ | \ i\in I\rangle$, then $G$ is \emph{generated by the $H_i$'s}.

\subsection{$\s$-Picard-Vessiot theory} \label{subsection PV}
We first recall some basic definitions and results from \cite{DiVizioHardouinWibmer:DifferenceGaloisTheoryOfLinearDifferentialEquations}. A \emph{$\de$-ring} is a ring $R$ together with a derivation $\de\colon R\to R$. An ideal $\ida\subseteq R$ such that $\de(\ida)\subseteq \ida$ is called a \emph{$\de$-ideal}. If every $\de$-ideal of $R$ is trivial, $R$ is called \emph{$\de$-simple}.
The \emph{$\de$-constants} of $R$ are $R^{\de}=\{r\in R \mid \de(r)=0\}$.

A \emph{$\ds$-ring} is a  ring $R$ together with a derivation $\de\colon R\to R$ and a ring endomorphism $\s\colon R\to R$ such that $\de(\s(r))=\hslash\s(\de(r))$ for all $r\in R$ for some fixed unit $\hslash\in R^\de$.
If $R$ is a field, we speak of a $\ds$-field. There are the obvious notions of a morphism of $\ds$-rings, of $\ds$-algebras etc.
A typical example of a $\ds$-field is the field $k(x)$ of rational functions over a $\s$-field $k$, considered as a $\ds$-field with $\de=\frac{d}{dx}$ and $\s\colon k(x)\to k(x)$ extending $\s\colon k\to k$ by $\s(x)=x$.
The $\ds$-field we are primarily interested in is the field $\C(x)$ with derivation $\de=\frac{d}{dx}$
and endomorphism $\s$ given by $\s(f(x))=f(x+1)$.

From now on let $F$ denote a $\ds$-field of characteristic zero and let $k=F^\de$ be the $\s$-field of $\de$-constants of $F$. We consider a linear differential equation $\de(y)=Ay$ with a matrix $A\in F^{n\times n}$.

\begin{Def} \label{def: sPVring}
	A \emph{$\s$-Picard-Vessiot ring\footnote{This definition differs from Definition \cite[Def. 1.2]{DiVizioHardouinWibmer:DifferenceGaloisTheoryOfLinearDifferentialEquations} where the condition $R^\de=k$ is dropped. Definition \ref{def: sPVring} is more convenient for us and \cite[Prop. 1.5]{DiVizioHardouinWibmer:DifferenceGaloisTheoryOfLinearDifferentialEquations} shows that with our definition, $\s$-Picard-Vessiot rings correspond to $\s$-Picard-Vessiot extensions.}} for $\de(y)=A y$ is an $F$-$\ds$-algebra $R$ such that
	\begin{enumerate}
		\item there exists a matrix $Y\in\GL_n(R)$ with $\de(Y)=AY$ and $R=F\{Y,1/\det(Y)\}$,
		\item $R$ is $\de$-simple and
		\item $R^\de=k$. 
	\end{enumerate} 
\end{Def}
The definition of a (classical) Picard-Vessiot ring for $\de(y)=Ay$ is identical to the above definition other than $R=F\{Y,1/\det(Y)\}$ replaced with $R=F[Y,1/\det(Y)]$. In practice, $\s$-Picard-Vessiot rings often arise as in the following lemma:

\begin{lem} \label{lemma: crit}
	Let $E/F$ be an extension of $\ds$-fields such that $E^\de=F^\de$. If $A\in F^{n\times n}$ and $Y\in \GL_n(E)$ are such that $\de(Y)=AY$, then $R=F\{Y,1/\det(Y)\}$ is a $\s$-Picard-Vessiot ring for $\de(y)=Ay$.
\end{lem}
\begin{proof}
	The field of fractions of $R$ is a $\s$-Picard-Vessiot extension for $\de(y)=Ay$ in the sense of \cite[Def. 1.2]{DiVizioHardouinWibmer:DifferenceGaloisTheoryOfLinearDifferentialEquations}. It thus follows from \cite[Prop. 1.5]{DiVizioHardouinWibmer:DifferenceGaloisTheoryOfLinearDifferentialEquations} that $R$ is a $\s$-Picard-Vessiot ring for $\de(y)=Ay$.
\end{proof}

A $\s$-Picard-Vessiot ring is an integral domain and $\s$ and $\de$ extend uniquely to the field of fractions of $R$. The field of fractions $E$ of a $\s$-Picard-Vessiot ring is a \emph{$\s$-Picard-Vessiot extension}. The \emph{$\s$-Galois group} $G$ of a $\s$-Picard-Vessiot ring $R/F$ is the functor from the category of \ks-algebras to the category of groups given by 
$$G(S)=\Aut^{\ds}(R\otimes_k S/F\otimes_k S)$$
for every \ks-algebra $S$. Here $R\otimes_k S$ is considered as a $\ds$-ring with $\de$ being the trivial derivation on $S$, i.e., $\de(s)=0$ for $s\in S$. The choice of a fundamental solution matrix $Y\in\GL_n(R)$ for $\de(y)=Ay$ determines a $\s$-closed embedding of $G$ into $\GL_n$. Indeed, for every \ks-algebra $S$ and $g\in G(S)$ there exists a matrix $\phi_S(g)\in\GL_n(S)$ such that $g(Y)=Y\phi_S(g)$. Then $\phi\colon G\to \GL_n$ is a $\s$-closed embedding. We let $\Gal_Y(R/F)$ denote the image $\phi(G)$ of $G$ in $\GL_n$ and call it the $\s$-Galois group of $R/F$ with respect to $Y$. The coordinate ring of $G$ is $k\{G\}=(R\otimes_F R)^\de$ and the canonical map 
\begin{equation} \label{eqn: alg torsor isom} R\otimes_k k\{G\}\to R\otimes_F R
\end{equation}
is an isomorphism.

{
\begin{ex}\label{ex PVR}
 Consider the differential equation $\de(y)=y$ over $\C(x)$ equipped with $\delta=d/dx$ and the shift operator $\sigma$, i.e., $\sigma(f(x))=f(x+1)$. Then $R=\C(x)\{e^x,e^{-x}\}$ is a $\s$-Picard-Vessiot ring for this equation by Lemma \ref{lemma: crit} applied to the field $E$ of meromorphic functions. We compute $\s(e^x)=e^{x+1}=e\cdot e^x$ and conclude $R=\C(x)[e^x,e^{-x}]$. Let $S$ be a $k$-$\s$-algebra and let $\gamma$ be an automorphism on $R\otimes_\C S$ of $(F\otimes_\C S)$-$\delta \sigma$-algebras. Then $\gamma$ is uniquely determined by $\gamma(e^x)$ and as $\gamma$ commutes with $\de$, there exists an $\alpha \in S^\times$ with $\gamma(e^x)=e^x\alpha$. Moreover, $\gamma$ commutes with $\sigma$ and hence $\sigma(\alpha)=\alpha$. Conversely, every $\alpha \in S$ with $\s(\alpha)=\alpha$ gives rise to such an automorphism (here we use that $e^x$ is transcendental over $\C(x)$). We conclude that the $\s$-Galois group of $R$ is the constant subgroup of the multiplicative group: $G(S)=\{\alpha \in S^\times \mid \sigma(\alpha)=\alpha\}$ for all \ks-algebras $S$. 
\end{ex}
}

Let $S$ be a \ks-algebra and $g\in G(S)$. Then $g\colon R\otimes_k S\to R\otimes_k S$ extends to an automorphism $\widetilde{g}\colon \Frac(R\otimes_k S)\to\Frac(R\otimes_k S)$ on the total ring of fractions of $R\otimes_k S$, which includes $E=\Frac(R)$. An element $a\in E$ is \emph{invariant} under $g$ if $\widetilde{g}(a)=a$. For a $\s$-closed subgroup $H$ of $G$, the set of all elements in $E$ that are invariant under $H(S)$ for every $k$-$\s$-algebra $S$ is denoted by $E^H$.

\begin{thm}[{$\s$-Galois correspondence, \cite[Theorem 3.2]{DiVizioHardouinWibmer:DifferenceGaloisTheoryOfLinearDifferentialEquations}}] \label{theo: Galois correspondence}
	The map $H\mapsto E^H$ is an inclusion reversing bijection between the $\s$-closed subgroups of $G$ and the intermediate $\ds$-field of $E/F$. In particular, $E^H=F$ if and only if $H=G$.
\end{thm}
The following proposition will be helpful for constructing explicit examples.

\begin{prop}[{\cite[Prop. 2.15]{DiVizioHardouinWibmer:DifferenceGaloisTheoryOfLinearDifferentialEquations}}] \label{prop: relation to clasical PV}
	Let $R=F\{Y,1/\det{Y}\}$ be a $\s$-Picard-Vessiot ring for $\de(y)=Ay$, where $A\in F^{n\times n}$. Then, $F[Y,\s(Y),\ldots,\s^i(Y),1/\det(Y\ldots\s^i(Y))]$ is a (classical) Picard-Vessiot ring with (classical) Galois group isomorphic to the $i$-th order Zariski closure of $\Gal_Y(R/F)$ in $\GL_n$ for every $i\geq 1$.
\end{prop}

\begin{cor}\label{cor: Zariski closure}
 	Let $R=F\{Y,1/\det{Y}\}$ be a $\s$-Picard-Vessiot ring for $\de(y)=Ay$ and define $R_i=F[Y,\s(Y),\ldots,\s^i(Y),1/\det(Y\ldots\s^i(Y))]$ for $i\geq 1$. Let further $\H\leq \GL_n$ be a connected algebraic group with $\Gal_Y(R/F)\leq [\s]_k\H$. Then $\Gal_Y(R/F)=[\s]_k\H$ if and only if $\dim(R_i)=(i+1)\dim(\H)$ for all $i\geq 1$. 
\end{cor}
\begin{proof}
Set $G=\Gal_Y(R/F)$. By Proposition \ref{prop: relation to clasical PV}, $R_i/F$ is a (classical) Picard-Vessiot ring with (classical) Galois group $G[i]$. Hence $\dim(R_i)=\dim(G[i])$ for all $i\geq 1$. If $G=[\s]_k\H$, then $G[i]=\H\times{\hs\H}\times\ldots\times{\hsi \H}\simeq\H^{i+1}$ and thus $\dim(R_i)=\dim(G[i])=(i+1)\dim(\H)$ for all $i\geq 1$. Conversely, assume that $G$ is properly contained in $[\s]_k\H$. Then the defining ideal $\mathfrak a \subseteq [\s]_k k[\H]$ of $G$ as a $\s$-closed subgroup of $[\s]_k\H$ is non-zero and so $\mathfrak a \cap k[\H][i]$ is non-zero for some $i\geq 1$. Hence $G[i]$ is properly contained in $\H\times{\hs\H}\times\ldots\times{\hsi \H}\simeq\H^{i+1}$ and as $\H^{i+1}$ is connected, we conclude $\dim(R_i)=\dim(G[i])<\dim(\H^{i+1})=(i+1)\dim(\H)$.
\end{proof}

For later use we record a lemma about the $\s$-Galois group with respect to a fundamental solution matrix.

\begin{lem}\label{lemma: compositum}
Let $R=F\{Y,1/\det(Y)\}$ be a $\s$-Picard-Vessiot ring for $\de(y)=Ay$. Then the following holds.
\begin{enumerate}
\item For every $B\in \GL_n(F)$, $R$ is a $\s$-Picard-Vessiot ring for $\de(y)=(BAB^{-1}+\de(B)B^{-1})y$ with fundamental solution matrix $BY$ and we have $\Gal_{BY}(R/F)=\Gal_{Y}(R/F)$ as $\s$-closed subgroups of $\GL_n$.
\item Let $F_0\supseteq F_1\supseteq F$ be $\ds$-field extensions with $F_0^\de=F$ and suppose that $R\subseteq F_0$ as an $F$-$\ds$-subalgebra. Then $R_1=F_1\{Y,1/\det(Y)\}\subseteq F_0$ is a $\s$-Picard-Vessiot ring over $F_1$ and $\Gal_{Y}(R_1/F_1)\leq \Gal_{Y}(R/F)$ as $\s$-closed subgroups of $\GL_n$.
\end{enumerate}
\end{lem}
\begin{proof}
For (i), define $\tilde Y=BY$. Then $\de(\tilde Y)\tilde Y^{-1}=\de(BY)Y^{-1}B^{-1}=BAB^{-1}+\de(B)B^{-1}$, since $\de(Y)=AY$. Also $R=F\{Y,1/\det(Y)\}=F\{\tilde Y,1/\det(\tilde Y)\}$. Hence $R$ is a $\s$-Picard-Vessiot ring for $\de(y)=(BAB^{-1}+\de(B)B^{-1})y$ with fundamental solution matrix $BY$. Let $S$ be a $k$-$\s$-algebra and let $g\in \Aut^\ds(R\otimes_k S/F\otimes_k S)$. Then $\tilde Y^{-1}g(\tilde Y)=Y^{-1}B^{-1}Bg(Y)=Y^{-1}g(Y)$ and thus the $S$-rational points of $\Gal_{BY}(R/F)$ and $\Gal_{Y}(R/F)$ coincide as subsets of $\GL_n(S)$ and (i) follows. 

For (ii), first note that $R_1$ is a $\s$-Picard-Vessiot ring over $F_1$ by Lemma \ref{lemma: crit}, since $F_0^\de=F_1^\de$. Let $S$ be a $k$-$\s$-algebra and let $g_1\in \Aut^\ds(R_1\otimes_k S/F_1\otimes_k S)$. Then $g_1(Y)=Y\phi_S(g_1)$ with $\phi_S(g_1) \in \GL_n(S)$. Hence $g_1(Y)$ has entries in $R\otimes_k S$ and thus $g_1$ restricts to an injective homorphism $g$ on $R\otimes_k S$. As $R\otimes_k S$ is generated by the entries of $Y=g(Y\phi_S(g_1)^{-1})$ and $1/\det(Y)=g(1/\det(Y\phi_S(g_1^{-1})))$, $g$ is surjective and thus an element of $\Aut^\ds(R\otimes_k S/F\otimes_k S)$. As $g_1$ is uniquely determined by $g_1(Y)=g(Y)$, we obtain an inclusion $\Aut^\ds(R_1\otimes_k S/F_1\otimes_k S)\hookrightarrow \Aut^\ds(R\otimes_k S/F\otimes_k S)$ that corresponds to a containment $\Gal_{Y}(R_1/F_1)(S)\subseteq \Gal_{Y}(R/F)(S)$ as subsets of $\GL_n(S)$, since $Y^{-1}g_1(Y)=Y^{-1}g(Y)$.
\end{proof}

\section{Multiplicative, additive, and finite cyclic groups as $\s$-Galois groups}
The patching method will allow to break down the task of realizing a given group $G$ as a $\s$-Galois group to the task of realizing generating subgroups as $\s$-Galois groups over certain overfields (see Theorem \ref{thm: patching}). As our task is to realize all algebraic groups as $\s$-Galois groups, we need to find generating subgroups that we can explicitly realize as $\s$-Galois groups. In this section, we first prove that every algebraic group can be generated (as a $\s$-algebraic group) by finitely many subgroups that are each either isomorphic to the multiplicative, or the additive, or a finite cyclic group. Then we show how these three types of groups can be realized as $\s$-Galois groups.  
\subsection{Generating algebraic groups as $\s$-algebraic groups}
It is known (see e.g., \cite[Prop. 3.1 ]{BachmayrHarbaterHartmann:DifferentialGaloisGroupsOverLaurentSeriesFields}) that every linear algebraic group over an algebraically closed field of characteristic zero can be generated by finitely many closed subgroups isomorphic to the additive group $\Ga$, the multiplicative group $\Gm$ or a finite cyclic group. We will need a slightly refined version of this result.

\begin{lem} \label{lemma:strong finite generation for algebraic groups}
	Let $\G$ be a linear algebraic group over an algebraically closed field  of characteristic zero. Then there exist closed subgroups $\H_1,\ldots,\H_r$ of $\G$ such that
	\begin{enumerate}
		\item each $\H_i$ is isomorphic to either $\Ga$, $\Gm$ or a finite cyclic group and
		\item  for some $n\geq 1$ and $(i_1,\ldots,i_n)\in\{1,\ldots,r\}^n$ the multiplication map \\ $\f\colon\X=\H_{i_1}\times\ldots\times \H_{i_n}\to \G$ is surjective.
	\end{enumerate}
\end{lem}
\begin{proof}
	It is shown in the proof of \cite[Prop. 3.1]{BachmayrHarbaterHartmann:DifferentialGaloisGroupsOverLaurentSeriesFields} that the identity component $\G^o$ of $\G$ can be generated by finitely many copies of $\Ga$'s and $\Gm$'s. According to \cite[Cor.~ 2.2.7]{Springer}, Condition (ii) above is automatically satisfied for generating subgroups that are connected. So we can find finitely many closed subgroups $\H_1,\ldots,\H_s$ of $\G^o$ isomorphic to either $\Ga$ or $\Gm$ such that the multiplication map from $\H_1\times\ldots\times\H_s$ to $\G^0$ is surjective.
	According to a result of Borel and Serre (\cite[Lemma 5.11]{BorelSerre}), there exists a finite subgroup $\mathcal{W}$ of $\G$ that meets every connected component of $\G$. Clearly $\mathcal{W}$, is generated by finite cyclic groups $\H_{s+1},\ldots,\H_r$, even in the strong sense of (ii). Thus the subgroups $\H_1,\ldots,\H_r$ satisfy condition (ii).
\end{proof}

We now transfer the above result to $\s$-algebraic groups.

\begin{prop}\label{prop: generating}
	Let $k$ be an algebraically closed $\s$-field of characteristic zero and let $\G$ be a linear algebraic group over $k$. Then the $\s$-algebraic group $G=[\s]_k\G$ can be generated by finitely many $\s$-closed subgroups isomorphic to $[\s]_k\H$, where $\H$ is $\Gm$, $\Ga$ or a finite cyclic group.
\end{prop}
\begin{proof}
	Let $\H_1,\ldots,\H_r$ and $\f\colon \X\to \G$ be as in Lemma \ref{lemma:strong finite generation for algebraic groups} and consider the morphism $[\s]_k\f\colon X\to G$ of $\s$-varieties, where $X=[\s]_k\X=[\s]_k\H_{i_1}\times\ldots\times  [\s]_k\H_{i_r}$. Then $[\s]_k\f$ is the multiplication map, hence $([\s]_k\f)(X)\subseteq  \langle [\s]_k\H_i|\ i=1,\ldots,r\rangle\subseteq G$. On the other hand, $([\s]_k\f)(X)$ is the $\s$-closed $\s$-subvariety of $G$ defined by the kernel of $([\s]_k\f)^*=[\s]_k\f^*\colon k\{G\}\to k\{X\}$. Since $\f^*\colon k[\G]\to k[\X]$ is injective, also $[\s]_k\f^*\colon k\{G\}\to k\{X\}$ is injective (Lemma \ref{lemma: sigmaization is injective})
	and so $([\s]_k\f)(X)=G$.
	Thus $G=\langle [\s]_k\H_i|\ i=1,\ldots,r\rangle$ as desired.
	\end{proof}

\subsection{Realization of $\mathbb G_m$, $\mathbb G_a$ and finite cyclic groups as $\s$-Galois groups}
In this subsection, we present criteria for when a $\s$-Picard-Vessiot ring has a $\s$-Galois group isomorphic to the additive group $\Ga$, the multiplicative group $\Gm$, or a finite cyclic group (all interpreted as $\s$-algebraic groups). We let $F$ denote a $\ds$-field of characteristic zero and we let $k=F^\de$ be its $\s$-field of $\de$-constants. 

We begin with the multiplicative case.

\begin{prop}\label{prop: Gm}
Let $L/F$ be a $\ds$-field extension with $L^\de=F^\de$. Let $y\in L$ be an element such that $\de(y)y^{-1}\in F$. Then $F\{y,y^{-1}\}\subseteq L$ is a $\s$-Picard-Vessiot ring with $\s$-Galois group $G$ isomorphic to a $\s$-closed subgroup of the multiplicative group $\Gm$. We have $G\simeq[\s]_k\Gm$ if and only if $y,\s(y),\s^2(y),\dots$ are algebraically independent over $F$. 
\end{prop}
\begin{proof}
Define $a=\de(y)y^{-1} \in F$. By Lemma \ref{lemma: crit}, $R=F\{y,y^{-1}\}$ is a $\s$-Picard Vessiot ring for the differential equation $\de(y)=ay$. Thus $\Gal_y(R/F)$ is a $\s$-closed subgroup of $\GL_1=\Gm$. Set $G=\Gal_y(R/F)$. 

For every $i\geq 1$, define $R_i=F[y,\s(y),\ldots,\s^i(y), 1/(y\ldots\s^i(y))]$. By Corollary \ref{cor: Zariski closure}, $G=[\s]_k\Gm$ if and only if $\dim(R_i)=i+1$ for all $i\geq 1$. Clearly, $\dim(R_i)=i+1$ holds if and only if $y,\s(y),\s^2(y),\dots\s^i(y)$ are algebraically independent and the claim follows.
\end{proof}
We next treat the additive case.

\begin{prop}\label{prop: Ga}
Let $L/F$ be a $\ds$-field extension with $L^\de=F^\de$. Let $y\in L$ be an element such that $\de(y)\in F$. Then $F\{y\}\subseteq L$ is a $\s$-Picard-Vessiot ring with $\s$-Galois group $G$ isomorphic to a $\s$-closed subgroup of the additive group $\Ga$. We have $G\simeq[\s]_k\Ga$ if and only if $y,\s(y),\s^2(y),\dots$ are algebraically independent over $F$. 
\end{prop}
\begin{proof}
Define $a=\de(y) \in F$ and define matrices
$$A= \left(\begin{array}{cc}
0 & a\\
0 & 0
\end{array}\right) \in F^{2\times 2}, \ Y= \left(\begin{array}{cc}
1 & y\\
0 & 1
\end{array}\right) \in \GL_2(R). $$
Then $R=F\{Y,1/\det(Y)\}=F\{y\}$ is a $\s$-Picard-Vessiot ring for the differential equation $\de(y)=Ay$ by Lemma \ref{lemma: crit}.   
The $\s$-Galois group of $R$ is isomorphic to a $\s$-closed subgroup of $\Ga$. Indeed, if $S$ is a \ks-algebra, $g\in\Aut^{\ds}(R\otimes_k S/F\otimes_kS)$ and $\f\colon G\to \GL_2$ is the $\s$-closed embedding associated with the choice of $Y$, then
$$\left(\begin{array}{cc}
1 & g(y)\\
0 & 1
\end{array}\right)=g(Y)=Y\f_S(g)=\left(\begin{array}{cc}
g_{11}+yg_{21} & g_{12}+yg_{22}\\
g_{21} & g_{22}
\end{array}\right), $$
where $$\f_S(g)=\left(\begin{array}{cc}
g_{11} & g_{12}\\
g_{21} & g_{22}
\end{array}\right)\in\GL_2(S).
$$
Thus, $g_{21}=0$ and $g_{11}=g_{22}=1$. This shows that $\f(G)=\Gal_Y(R/F)$ is contained in the $\s$-closed subgroup $H\simeq [\s]_k\Ga$ of $\GL_2$ defined by $$H(S)=\left\{ \left(\begin{array}{cc}
1 & s\\
0 & 1
\end{array}\right)\ \Big|\ s\in S\right\}$$ 
for any \ks-algebra $S$.

For every $i\geq 1$, consider the $F$-$\de$-algebra $$R_i=F[y,\s(y),\ldots,\s^i(y)]=F[Y,\s(Y),\ldots,\s^i(Y),1/\det(Y\ldots\s^i(Y))].$$ By Corollary \ref{cor: Zariski closure}, $\Gal_Y(R/F)=H$ holds if and and only if $\dim(R_i)=i+1$ for all $i\geq 1$. Clearly, $\dim(R_i)=i+1$ holds if and only if $y,\s(y),\s^2(y),\dots, \s^i(y)$ are algebraically independent and the claim follows. 
\end{proof}

Finally, we treat the case of a finite cyclic group.

\begin{prop}\label{prop: cyclic}
Let $L/F$ be a $\ds$-field extension with $L^\de=F^\de$ algebraically closed. Let $y\in L$ be such that $y^d \in F$ for some $d\in \N$. Then $F\{y,y^{-1}\}\subseteq L$ is a $\s$-Picard-Vessiot ring with $\s$-Galois group $G$ isomorphic to a $\s$-closed subgroup of $\H$, where $\H$ is the finite cyclic group of order $d$ (considered as an algebraic group over $k$). We have $G\simeq[\s]_k\H$ if and only if for every $i\in \N$, $\s^i(y)$ has degree $d$ over 
$F(y,\s(y),\dots,\s^{i-1}(y))$.  
\end{prop}
\begin{proof}
	Let $b=y^d\in F$. Then $\de(y)=\frac{\de(b)}{bd}y$ and it follows from Lemma \ref{lemma: crit} that $R=F\{y,y^{-1}\}=F(y,\s(y),\ldots)$ is a $\s$-Picard-Vessiot ring. For every \ks-algebra $S$ and $g\in\Gal_y(R/F)(S)\leq \GL_1(S)$ we have $b=g(b)=g(y^d)=y^dg^d=bg^d$ and it follows that $g^d=1$. Thus $\Gal_y(R/F)$ is a $\s$-closed subgroup of $\mu_d$, the algebraic group of $d$-th roots of unity over $k$. Since $k$ is algebraically closed, $\mu_d$ is isomorphic to $\H$. So $G$ is isomorphic to a $\s$-closed subgroup of $\H$.
	
	Assume that $\Gal_y(R/F)=[\s]_k\mu_d$. As $(\s^i(y))^d=\s^i(b)$, $\s^i(y)$ has degree at most $d$ over $F(y,\s(y),\dots,\s^{i-1}(y))$. For $i\geq 0$, $R_i=F[y,\ldots,\s^i(y)]=F(y,\ldots,\s^i(y))$ is a (classical) Picard-Vessiot ring with (classical) Galois group isomorphic to $\mu_d^{i+1}$ by Proposition~\ref{prop: relation to clasical PV}. By the (classical) torsor theorem, $R_i\otimes_F R_i\cong R_i\otimes_k k[\mu_d^{i+1}]$.  Hence the dimension of $R_i$ as an $F$-vector space agrees with the dimension of $k[\mu_d^{i+1}]$ as a $k$-vector space, and the latter is $d^{i+1}$. Thus $[R_1:R_0][R_2:R_1]\cdots [R_i:R_{i-1}]=d^{i+1}$ and clearly, the field extension $R_j/R_{j-1}$ has degree at most $d$ for every $j\geq 0$, since $(\s^j(y))^d=\s^j(b)$. Therefore $[R_i:R_{i-1}]=d$ and $\s^i(y)$ has degree $d$ over $F(y,\s(y),\dots,\s^{i-1}(y))$.
	
	Conversely, if $F[y,\ldots,\s^i(y)]$ has $F$-dimension $d^{i+1}$ for every $i\geq 1$, then the $i$-th order Zariski closure of $\Gal_y(R/F)$ in $\GL_1$ must equal $\mu_d\times\ldots\times{\hsi\mu_d}$ and so $\Gal_y(R/F)$ must equal $[\s]_k\mu_d$.
\end{proof}

\section{Patching and $\s$-Picard-Vessiot theory}
In this section, we prove a patching result for $\s$-Picard-Vessiot rings that allows us to glue together two $\s$-Picard-Vessiot rings under certain assumptions. The proof relies on factorization of the fundamental solution matrices, so it is crucial that the $\s$-Galois groups of the given $\s$-Picard-Vessiot rings are embedded in $\GL_n$ for the same $n$. Therefore, we start the section by showing that the represenation of a $\s$-Galois group can be changed without changing the $\s$-Picard-Vessiot ring.   

\subsection{Changing the representation of a $\s$-Galois group}
Let $F$ be a $\ds$-field of characteristic zero and $R$ a $\s$-Picard-Vessiot ring over $F$ with $\s$-Galois group $G$. The choice of matrices $A\in F^{n\times n}$ and $Y\in \GL_n(R)$ such that $R$ is a $\s$-Picard-Vessiot ring for $\de(y)=Ay$ with fundamental solution matrix $Y$, determines a $\s$-closed embedding of $G$ into $\GL_n$ via $G\to \Gal_{Y}(R/F)$. The following proposition shows that, conversely, if we start with a $\s$-closed embedding of $G$ into $\GL_n$, then we can find appropriate matrices $A$ and $Y$.

\begin{prop}\label{prop: Tannaka}
	Let $R/F$ be a $\s$-Picard Vessiot ring with $\s$-Galois group $G$ and let $G'$ be a $\s$-closed subgroup of $\GL_n$ isomorphic to $G$. Then there exist $A\in F^{n\times n}$ and $Y\in \GL_n(R)$ with $\de(Y)=AY$ such that 
	$R$ is a $\s$-Picard-Vessiot ring for $\de(y)=Ay$ and $\Gal_{Y}(R/F)=G'$.
\end{prop}
\begin{proof}
	Let $X$ denote the $\s$-variety over the $\s$-field $F$ represented by the $F$-$\s$-algebra $R$. Then $X$ is naturally equipped with a right action $\f\colon X\times G_F\to X$ of the $\s$-algebraic group $G_F$ over $F$:
	For $x\in X(S)$, i.e., $x\colon R\to S$ is a morphism of $F$-$\s$-algebras, and $g\in G_F(S)=G(S)$ the element $\f(x,g)\in X(S)$ is the composition $$R\to R\otimes_k S\xrightarrow{g}R\otimes_k S\to S,$$ where the first map is the inclusion into the first factor and the last map sends $r\otimes s$ to $x(r)s$. The morphism $X\times G_F\to X\times X$, given by $(x,g)\mapsto(x,\f(x,g))$ for $x\in X(S),\ g\in G_F(S)$ for any $F$-$\s$-algebra $S$, is an isomorphism, since the dual map $R\otimes_F R\simeq R\otimes_F F\{G_F\}=R\otimes_k k\{G\}$ is the isomorphism (\ref{eqn: alg torsor isom}). Thus $X$ is a right $G_F$-torsor. 
	In \cite[Lemma 4.4]{BachmayrWibmer:TorsorsForDifferenceAlgebraicGroups} it is shown that every left torsor for a $\s$-closed subgroup of $\GL_n$ is isomorphic to a $\s$-closed $\s$-subvariety of $\GL_n$, with action given by matrix multiplication. If $\widetilde{X}$ is a right $\widetilde{G}$-torsor with action $\widetilde{\f}\colon \widetilde{X}\times \widetilde{G}\to \widetilde{X}$, then $\widetilde{X}$ is a left torsor with action $\widetilde{G}\times \widetilde{X}\to \widetilde{X},\ (g,x)\mapsto\f(x,g^{-1})$. Therefore, it follows that there exists a $\s$-closed embedding $\varphi\colon X\to \GL_n$ such that the right action of $G_F$ on $X$ is given by matrix multiplication, where $G$ is identified with the $\s$-closed subgroup $G'$ of $\GL_n$.
	The $\s$-closed embedding $\varphi\colon X\to \GL_n$ corresponds to a surjective morphism $\varphi^*\colon F\{Z,1/\det{Z}\}\to R$ of $F$-$\s$-algebras, where $Z$ is an $n\times n$ matrix of $\s$-indeterminates over $F$. Let $Y\in \GL_n(R)$ denote the image of $Z$.
	
	For a \ks-algebra $S$, the action of an element $g\in G(S)$ on $Y$ is given by $g(Y)=Yg$, where $g$ is considered as an element of $\GL_n(S)$ via the isomorphism $G\to G'$. Therefore, $$g(\de(Y)Y^{-1})=\de(g(Y))g(Y)^{-1}=\de(Yg)(Yg)^{-1}=\de(Y)gg^{-1}Y=\de(Y)Y.$$ By 
	Theorem \ref{theo: Galois correspondence} we have $A=\de(Y)Y^{-1}\in F^{n\times n}$ and we see that $R$ is a $\s$-Picard-Vessiot ring for $\de(y)=Ay$ with fundamental solution matrix $Y$ and $\Gal_Y(R/F)=G'$.	
\end{proof}

\subsection{A $\ds$-setup for applying patching methods}
The method of patching over fields was introduced by Harbater and Hartmann in \cite{HH}. In the most basic setup, they consider fields $F\subseteq F_1,F_2$ with a common overfield $F_0$ and show that finite dimensional vector spaces $V_1$ over $F_1$ and $V_2$ over $F_2$ together with an isomorphism $\psi \colon V_1\otimes_{F_1}F_0\to V_2\otimes_{F_2}F_0$ can be glued together to an $F$-vector space $V$ with isomorphisms $\phi_1\colon V\otimes_F F_1 \simeq V_1$ and $\phi_2\colon V\otimes_F F_2 \simeq V_2$ such that $(\phi_2\otimes_{F_2}F_0)\circ (\phi_1\otimes_{F_1}F_0)^{-1}=\psi$  if $(F,F_1,F_2,F_0)$ satisfies two properties called \textit{intersection} and \textit{factorization}. We call such quadruples \textit{diamonds with the factorization property}. In \cite{BachmayrHarbaterHartmannWibmer:DifferentialEmbeddingProblemsOverComplexFunctionFields}, we considered \textit{differential} diamonds with the factorization property.
   
\begin{Def}
A \emph{diamond with the factorization property} is a quadruple $(F,F_1,F_2,F_0)$ of fields with inclusions $F\subseteq F_1,F_2$ and $F_1,F_2\subseteq F_0$ such that $F_1\cap F_2=F$ (intersection) and such that for every $n\in \N$ and every matrix $A\in \GL_n(F_0)$, there exist matrices $B \in \GL_n(F_1)$ and $C\in \GL_n(F_2)$ with $A=B\cdot C$ (factorization). If in addition $\operatorname{char}(F)=0$ and all of the fields are equipped with a derivation that is compatible with the inclusions, then $(F,F_1,F_2,F_0)$ is called a \emph{differential diamond with the factorization property}. Similarly, if all four fields are $\ds$-fields of characteristic zero such that both $\de$ and $\s$ are compatible with the inclusions, then we call $(F,F_1,F_2,F_0)$ a \emph{$\ds$-diamond with the factorization property}.
\end{Def}

Over $F=\C(x)$, diamonds with the factorization property arise as fields of meromorphic functions on suitabel open subsets of $\X=\mathbb P^1 _\C$, as the following lemma states.

\begin{lem}[\cite{BachmayrHarbaterHartmannWibmer:DifferentialEmbeddingProblemsOverComplexFunctionFields}, Lemma~3.4] \label{lemma: meromorphic diamond}
Let $F$ be a one-variable function field over $\C$, or equivalently the field of meromorphic functions on a compact Riemann surface $\X$.  
Let $O_1,O_2$ be connected metric open subsets of $\X$ such that $O_i \ne \X$, $O_1 \cup O_2 = \X$, and $O_0 := O_1 \cap O_2$ is connected.  Let $F_i$ be the field of meromorphic functions on 
$O_i$.  Then $(F,F_1,F_2,F_0)$ is a diamond with the factorization property. 
\end{lem}

\begin{ex}
 Let $\X=\mathbb P ^1_\C$ be the Riemann sphere and $F=\C(x)$. Consider the open subsets $O_1=\{ x \in \mathbb P ^1 _\C \mid |x|<1\}$ and $O_2=\{ x \in \mathbb P ^1 _\C \mid |x|>\frac{1}{2}\}$. Then $O_1 \cup O_2 = \X$ and the annulus $O_0= O_1 \cap O_2=\{ x \in \mathbb P ^1 _\C \mid \frac{1}{2}<|x|<1\}$ is connected. Let $F_i$ be the field of meromorphic functions on $O_i$.  Then $(F,F_1,F_2,F_0)$ is a diamond with the factorization property by Lemma \ref{lemma: meromorphic diamond}.
\end{ex}

\begin{cor}\label{cor: meromorphic diamond}
Let $U_1\supseteq U_2\supseteq U_3\supseteq \cdots$ and $V_1\supseteq V_2\supseteq V_3\supseteq \cdots$ be chains of proper connected metric open subsets of the Riemann sphere $\X=\mathbb P^1_\C$ such that $U_r\cap V_r$ is connected and $U_r\cup V_r=\X$ for all $r\in \N$. Let $F_{U_r}$, $F_{V_r}$ and $F_{U_r\cap V_r}$ denote the fields of meromorphic functions on $U_r$, $V_r$ and $U_r\cap V_r$, respectively, and let $F=\C(x)$ (the field of meromorphic functions on $\X$). Define $F_1=\varinjlim F_{U_r}$ (with respect to the natural inclusions $F_{U_r}\to F_{U_s}$ for $r\leq s$) and similarly $F_2=\varinjlim F_{V_r}$ and $F_0=\varinjlim F_{U_r\cap V_r}$. Then there are natural inclusions $F_1,F_2\subseteq F_0$ and $(F,F_1,F_2,F_0)$ is a diamond with the factorization property.
\end{cor}
\begin{proof}
By Lemma \ref{lemma: meromorphic diamond}, $(F,F_{U_r},F_{V_r}, F_{U_r\cap V_r})$ is a diamond with the factorization property for every $r\in \N$. Let $f\in F_0$ be contained in the intersection $F_1\cap F_2$. Then there exists an $r\in \N$ with $f \in F_{U_r}\cap F_{V_r}=F$. Thus $F_1\cap F_2=F$. Similarly, for $n \in \N$ and $A\in \GL_n(F_0)$ there exists an $r\in \N$ with $A\in \GL_n(F_{U_r\cap V_r})$. Hence there exist matrices $B\in \GL_n(F_{U_r})\subseteq \GL_n(F_1)$ and $C\in \GL_n(F_{V_r})\subseteq \GL_n(F_2)$ with $A=B\cdot C$.  
\end{proof}

\begin{ex}
 Let $\X=\mathbb P ^1_\C$ be the Riemann sphere and $F=\C(x)$. Consider the open subsets $U_r=\{ x \in \mathbb P ^1 _\C \mid |x|<1\}$ and $V_r=\{ x \in \mathbb P ^1 _\C \mid |x|>1-\frac{1}{r}\}$. Then for all $r\in \N$, $U_r \cup V_r = \X$ and the annulus $U_r \cap V_r=\{ x \in \mathbb P ^1 _\C \mid 1-\frac{1}{r}<|x|<1\}$ is connected. Define $F_1=\varinjlim F_{U_r}$, $F_2=\varinjlim F_{V_r}$ and $F_0=\varinjlim F_{U_r\cap V_r}$. Then $F_1=F_{U_1}$ is the field of meromorphic functions on the open set $\{ x \in \mathbb P ^1 _\C \mid |x|<1\}$, $F_2$ is the field of meromorphic functions on the open set $\{ x \in \mathbb P ^1 _\C \mid |x|>1\}$ that are meromorphic in $x=1$ and $F_0$ is the field of functions that are meromorphic in $x=1$.
\end{ex}

We proceed with a lemma that equips fields ocurring as direct limits as in Corollary \ref{cor: meromorphic diamond} with a structure of $\ds$-fields. 
\begin{lem}\label{lemma: direct limit}
Let $U_1\supseteq U_2\supseteq U_3\supseteq \cdots$ be a chain of connected metric open subsets of the Riemann sphere $\X=\mathbb P^1_\C$ and define $F_1=\varinjlim F_{U_r}$. Then the natural derivations $\de$ on $F_{U_r}$ induce a derivation $\de$ on $F_1$.  If in addition there are homomorphisms  $\sigma_r\colon F_{U_r}\to F_{U_{r+1}}$ for all $r\in \N$ such that
$$ \xymatrix{
F_{U_r} \ar^{\s_r}[r] \ar_\de[d] & F_{U_{r+1}} \ar^\de[d] \\
F_{U_r} \ar^{\s_r}[r] & F_{U_{r+1}}
} \quad \quad \text{ and } \quad \quad
\xymatrix{
F_{U_{r+1}} \ar^{\s_{r+1}}[r] & F_{U_{r+2}}  \\
F_{U_r} \ar^{\s_r}[r] \ar[u] & F_{U_{r+1}} \ar[u]
}
$$
commute, then $F_1$ is a $\ds$-field with $F_1^\de=\C$.
\end{lem}
\begin{proof}
It is immediate that the derivations on $F_{U_r}$ induce a well-defined derivation on the direct limit and also that the morphisms $\{\s_r\}_{r\geq 1}$ induce a well-defined endomorphism $\s\colon F_1 \to F_1$ that commutes with $\de$. Finally, $F_{U_r}^\de=\C$ for all $r$ and thus $F_1^\de=\C$.
\end{proof}

In the proof of our main theorem, we will work with the following kind of fields.
\begin{ex}\label{ex direct limit}
Fix a natural number $m\in \N$. 
\begin{enumerate}
\item For $r\in \N$, let $V_{r,m}=\{x \in \mathbb P^1_\C \mid x \notin [-r,0]+mi\}$ be the connected metric open subset of the Riemann sphere $\mathbb P^1_\C$ obtained by deleting the translate $[-r,0]+mi$ of the real interval $[-r,0]$. Note that $V_{1,m}\supseteq V_{2,m}\supseteq V_{3,m}\supseteq\dots$. We let $F_{V_{r,m}}$ denote the field of meromorphic functions on $V_{r,m}$. Then $x\mapsto x+1$ defines a holomorphic function $V_{r+1,m}\to V_{r,m}$ that induces a homomorphism $\sigma_r\colon F_{V_{r,m}}\to F_{V_{r+1,m}}$ for all $r\in \N$. Moreover, $\sigma_r$ is compatible with the natural derivations $\delta$ on $F_{V_{r,m}}$ and $F_{V_{r+1,m}}$ and $\sigma_{r+1}\colon F_{V_{r+1,m}}\to F_{V_{r+2,m}}$ restricts to $\sigma_r\colon F_{V_{r,m}}\to F_{V_{r+1,m}}$. Hence the direct limit
$$L(m)=\varinjlim\limits_{r\in \N} F_{V_{r,m}}$$ is a $\ds$-field with $F(m)^\delta=\C$ by Lemma \ref{lemma: direct limit}.
 \item For $r\in \N$, consider the subset  
$$U_{r,m}=\{x \in \mathbb P^1_\C \mid x \notin \bigcup\limits_{l=1}^m\left( [-r,0]+li\right)\}$$ of the Riemann sphere $\X=\mathbb P^1_\C$, where we delete the translates $[-r,0]+i$, $[-r,0]+2i$,\dots, $[-r,0]+mi$ of the real interval $[-r,0]$. Note that $U_{r,m}$ is a connected metric open subset of $\X$ and $U_{1,m}\supseteq U_{2,m}\supseteq U_{3,m}\supseteq \dots$. Let $F_{U_{r,m}}$ denote the field of meromorophic functions on $U_{r,m}$. Then $x\mapsto x+1$ defines a holomorphic function $U_{r+1,m}\to U_{r,m}$ that induces a homomorphism $\sigma_r\colon F_{U_{r,m}}\to F_{U_{r+1,m}}$ for all $r\in \N$. Moreover, $\sigma_r$ is compatible with the natural derivations $\delta$ on $F_{U_{r,m}}$ and $F_{U_{r+1,m}}$ and $\sigma_{r+1}\colon F_{U_{r+1,m}}\to F_{U_{r+2,m}}$ restricts to $\sigma_r\colon F_{U_{r,m}}\to F_{U_{r+1,m}}$. Hence the direct limit $$F(m)=\varinjlim\limits_{r\in \N} F_{U_{r,m}}$$  is a $\ds$-field with $F(m)^\delta=\C$ by Lemma \ref{lemma: direct limit}. 
\end{enumerate}
\end{ex}

\begin{rem}
 Note that the $\s$-fields $L(m)$ and $F(m)$ are not inversive, i.e., $\sigma$ is not surjective. For example, the image of $\s$ contains no functions with a singularity at the point $(-1/2+mi)$. Assume to the contrary that there exists an $f\in L(m)$ or $f\in F(m)$ such that $f(x+1)$ has a singularity at $(-1/2+mi)$. Then $f$ has a singularity at $1/2+mi$, a contradiction.
\end{rem}

Using Lemma \ref{lemma: direct limit}, we can now equip diamonds with the factorization property as in Corollary \ref{cor: meromorphic diamond} with a $\ds$-structure.

\begin{cor}\label{cor: diamond}
Let $U_1\supseteq U_2\supseteq U_3\supseteq \cdots$ and $V_1\supseteq V_2\supseteq V_3\supseteq \cdots$ be chains of proper connected metric open subsets of the Riemann sphere $\X=\mathbb P^1_\C$ such that $U_r\cap V_r$ is connected and $U_r\cup V_r=\X$ for all $r\in \N$. Let $F_{U_r}$, $F_{V_r}$ and $F_{U_r\cap V_r}$ denote the fields of meromorphic functions on $U_r$, $V_r$ and $U_r\cap V_r$, respectively, and let $F=\C(x)$ denote the field of meromorphic functions on $\X$. Define $F_1=\varinjlim F_{U_r}$ (with respect to the natural inclusions $F_{U_r}\to F_{U_s}$ for $r\leq s$) and similarly $F_2=\varinjlim F_{V_r}$ and $F_0=\varinjlim F_{U_r\cap V_r}$. 
For every $r\geq 1$ let $\sigma_r\colon F_{U_r\cap V_r}\to F_{U_{r+1}\cap V_{r+1}}$ be a morphism such that $\s_r(F_{U_r})\subseteq F_{U_{r+1}}$, $\s_r(F_{V_r})\subseteq F_{V_{r+1}}$ and
$$ \xymatrix{
	F_{U_r\cap V_r} \ar^-{\s_r}[r] \ar_\de[d] & F_{U_{r+1}\cap V_{r+1}} \ar^\de[d] \\
	F_{U_r\cap V_r} \ar^-{\s_r}[r] & F_{U_{r+1}\cap V_{r+1}}
} \quad \quad \text{ and } \quad \quad
\xymatrix{
	F_{U_{r+1}\cap V_{r+1}} \ar^{\s_{r+1}}[r] & F_{U_{r+2}\cap V_{r+2}}  \\
	F_{U_r\cap V_r} \ar^{\s_r}[r] \ar[u] & F_{U_{r+1}\cap V_{r+1}} \ar[u]
}
$$
commute. Then $(F,F_1,F_2,F_0)$ is a $\ds$-diamond with the factorization property.
\end{cor}
\begin{proof}
By Corollary \ref{cor: meromorphic diamond}, $(F,F_1,F_2,F_0)$ is a diamond with the factorization property. For all $r\in \N$, the fields $F_{U_r}$ are differential fields and the derivation is compatible with the inclusions $F_{U_r}\subseteq F_{U_s}$ for $r\leq s$, hence we can equip $F_1$ with a derivation and similarly for $F_2$ and $F_0$. These derivations are compatible with the natural inclusions $F\subseteq F_1,F_2$ and $F_1,F_2 \subseteq F_0$ and thus $(F,F_1,F_2,F_0)$ is a differential diamond with the factorization property. The morphisms $\{\s_r\}_{r\geq 1}$ induce a well-defined endomorphism $\s\colon F_0 \to F_0$ that restricts to endomorphisms on $F_1$ and $F_2$. Moreover, $\de(\sigma(f))=\s(\de(f))$ for all $f\in F_{0}$ and $\s$ restricts to an endomorphism of $F$ because $F_{U_r}\cap F_{V_r}=F$. Hence $(F,F_1,F_2,F_0)$ is a $\ds$-diamond with the factorization property. 
\end{proof}

Using Corollary \ref{cor: diamond}, we can now use the fields from Example \ref{ex direct limit} to obtain $\ds$-diamonds $(F,F(m-1),L(m),F(m))$ for all $m\in \N$:
\begin{ex}\label{ex diamond} Consider $F=\C(x)$, $\X=\mathbb P ^1_\C$ and fix a natural number $m \in \N$. For $r\in \N$, let $U_{r,m}$, $V_{r,m}$ be as defined in Example \ref{ex direct limit}. Hence $U_{r,m-1}\cap V_{r,m}=U_{r,m}$ is connected and $U_{r,m-1}\cup V_{r,m}=\X$ for all $r\in \N$. 

Define 
\begin{eqnarray*}
F_1&=&\varinjlim\limits_{r\in\N} F_{U_{r,m-1}}=F(m-1) \\
F_2&=&\varinjlim\limits_{r\in\N}  F_{V_{r,m}}=L(m) \\ 
F_0&=&\varinjlim\limits_{r\in\N}  F_{U_{r,m-1}\cap V_{r,m}}=F(m).
\end{eqnarray*}
 By Example \ref{ex direct limit} together with Corollary \ref{cor: diamond}, $(F,F_1,F_2,F_0)$ is a $\ds$-diamond with the factorization property with $F_0^\de=\C=F^\de$.
\end{ex}

\subsection{A patching result for $\s$-Picard-Vessiot rings}

The next theorem allows us to glue together two $\s$-Picard-Vessiot rings under certain assumptions. The proof is similar to related statements for Picard-Vessiot rings (\cite[Thm. 2.4]{BachmayrHarbaterHartmann:DifferentialGaloisGroupsOverLaurentSeriesFields}) and parameterized Picard-Vessiot rings (\cite[Thm. 2.2]{param_LAG}).

\begin{thm}\label{thm: patching}
Let $(F,F_1,F_2,F_0)$ be a $\ds$-diamond with the factorization property such that $k:=F^\de=F_0^\de$. Let $G$ be a $\s$-algebraic group over $k$ generated by two $\s$-closed subgroups $H_1$ and $H_2$. For $i=1,2$, let $R_i/F_i$ be a $\s$-Picard-Vessiot ring with $\s$-Galois group isomorphic to $H_i$ such that $R_i\subseteq F_0$ as an $F_i$-$\ds$-subalgebra. Then there exists a $\s$-Picard-Vessiot ring $R/F$ with $\s$-Galois group isomorphic to $G$ and $R\subseteq F_0$ as an $F$-$\ds$-subalgebra.
\end{thm}
\begin{proof}
According to Lemma \ref{lemma: linearization}, we can identify $G$ with a $\s$-closed subgroup of $\GL_n$ for a suitable $n \in \N$ and thus also view $H_1$ and $H_2$ as $\s$-closed subgroups of $\GL_n$. By Proposition~\ref{prop: Tannaka}, for $i=1,2$, there exists a  differential equation $\de(y)=A_iy$ with $A_i \in F_i^{n\times n}$ together with a fundamental solution matrix $Y_i \in \GL_n(R_i)$ such that $\Gal_{Y_i}(R_i/F_i)=H_i$. Since $R_i\subseteq F_0$, we can consider the matrix $Y_1Y_2^{-1} \in \GL_n(F_0)$. The factorization property implies that there exist matrices $B_1\in \GL_n(F_1)$, $B_2\in \GL_n(F_2)$ with $Y_1Y_2^{-1}=B_1^{-1}B_2$. Define $Y=B_1Y_1=B_2Y_2 \in \GL_n(F_0)$ and $A=\de(Y)Y^{-1} \in F_0^{n\times n}$. Then for both $i=1,2$,
\begin{eqnarray*}
A&=&\de(B_iY_i)Y_i^{-1}B_i^{-1} \\
&=& \de(B_i)B_i^{-1}+B_iA_iB_i^{-1} \in F_i^{n\times n}.
\end{eqnarray*}
As $F_1\cap F_2=F_0$, we conclude that $A$ has entries in $F$. Consider the differential equation $\de(y)=Ay$ over $F$ and define $R=F\{Y,1/\det(Y)\}\subseteq F_0$. Since $F_0^\de=F^\de$, $R$ is a $\s$-Picard-Vessiot ring over $F$ for $\de(y)=Ay$ by Lemma \ref{lemma: crit}. Define $H=\Gal_Y(R/F)$. We claim that $H=G$ as $\s$-closed subgroups of $\GL_n$. For $i=1,2$, $R_i=F_i\{Y_i,1/\det(Y_i)\}=F_i\{Y,1/\det(Y)\}$ and $$H_i=\Gal_{Y_i}(R_i/F_i)=\Gal_{B_iY_i}(R_i/F_i)= \Gal_{Y}(R_i/F_i)$$ by Lemma \ref{lemma: compositum}.(i). Furthermore, $\Gal_{Y}(R_i/F_i)\leq \Gal_{Y}(R/F)$ by Lemma \ref{lemma: compositum}.(ii) as $\s$-closed subgroups of $\GL_n$. Hence $H_1,H_2\leq H$ inside $\GL_n$ and thus $G\leq H$. 

 For both $i=1,2$, 
$$\Frac(R)^G\subseteq \Frac(R)^{H_i}\subseteq \Frac(R_i)^{H_i}=F_i,$$ where the last equality follows from Theorem \ref{theo: Galois correspondence}. Hence $\Frac(R)^G\subseteq F_1\cap F_2=F$ and therefore $G=H$ by Theorem~\ref{theo: Galois correspondence}.
\end{proof}

\section{Building blocks}
Theorem \ref{thm: patching} breaks down the problem of realizing a given group as $\sigma$-Galois group over $F=\C(x)$ to the task of realizing suitable \textit{building blocks}, i.e., realizing suitable, generating subgroups as $\sigma$-Galois groups over certain overfields of $F$. In this section, we explicitly construct these building blocks. 
\subsection{Independence of translates.}
In this subsection, we provide results on the algebraic or linear independence of translates of certain exponential, logarithmic and root functions as a preparation for constructing $\s$-Picard-Vessiot rings with $\s$-Galois groups isomorphic to $[\s]_\C\H$ for $\H=\Gm, \Ga$, or a finite cyclic group.

The following lemma corresponds to the case $\H=\Gm$.
\begin{lem}\label{lemma: 1}
For $\gamma \in \C$, the (infinite) set of functions
$$\exp\left(\frac{1}{x-\gamma}\right),\exp\left(\frac{1}{x+1-\gamma}\right),\exp\left(\frac{1}{x+2-\gamma}\right),\dots $$
is algebraically independent over $\C(x)$.
\end{lem} 
\begin{proof}
Assume to the contrary that the set is algebraically dependent. Then there exists an $m\in \N$ such that $\exp\left(\frac{1}{x-\gamma}\right),\exp\left(\frac{1}{x+1-\gamma}\right),\dots,\exp\left(\frac{1}{x+m-\gamma}\right)$ are algebraically dependent. By the theorem of Kolchin-Ostrowski (see e.g. Section 2 in \cite{MR240106}), this can only happen if there exist a non-zero vector $(e_0,\dots,e_m) \in \Z^{m+1}$ with 
$$f:= \prod \limits_{j=0}^m \left( \exp\left(\frac{1}{x+j-\gamma}\right)\right)^{e_j} \in \C(x).$$ We compute 
$$\frac{\de f}{f}=\sum\limits_{j=0}^m \frac{-e_j}{(x+j-\gamma)^2}.$$ On the other hand, we can factor $f=\alpha\prod_{j=1}^r(x-\alpha_j)^{n_j}$ for suitable $\alpha \in \C$, pairwise distinct elements $\alpha_j \in \C$ and $n_j\in \Z$. Thus 
$$\frac{\de f}{f}=\alpha \sum\limits_{j=1}^r \frac{n_j}{(x-\alpha_j)},$$
a contradiction.
\end{proof}

We next treat the case $\H=\Ga$.

\begin{lem}\label{lemma: 2}
For $\gamma \in \C$, the (infinite) set of functions
$$\log \left(\frac{1}{x-\gamma}+1\right),\log\left(\frac{1}{x+1-\gamma}+1\right),\log\left(\frac{1}{x+2-\gamma}+1\right),\dots $$
is algebraically independent over $\C(x)$.
\end{lem} 
\begin{proof}
Assume to the contrary that the set is algebraically dependent. Then there exists an $m\in \N$ such that $\log \left(\frac{1}{x-\gamma}+1\right), \log\left(\frac{1}{x+1-\gamma}+1\right),\dots,\log\left(\frac{1}{x+m-\gamma}+1\right)$ are algebraically dependent. By the theorem of Kolchin-Ostrowski (see e.g. Section 2 in \cite{MR240106}), this can only happen if there exist a non-zero vector $(c_0,\dots,c_m) \in \C^{m+1}$ with 
$$f:= \sum \limits_{j=0}^m c_j \log\left(\frac{1}{x+j-\gamma}+1\right) \in \C(x).$$ We compute 
\begin{eqnarray*}
\de(f)&=&\sum\limits_{j=0}^m \frac{-c_j}{(x+j-\gamma)(x+j+1-\gamma)}=\sum\limits_{j=0}^m \left(\frac{-c_j}{(x+j-\gamma)}+\frac{c_j}{(x+j+1-\gamma)}\right).
\end{eqnarray*}  
As $f$ is contained in $\C(x)$, all terms need to cancel. As the terms with denominator $(x-\gamma)$ and $(x+m+1-\gamma)$ cannot cancel with any other term, we conclude $c_0=c_m=0$. Inductively, we obtain that $c_j=0$ for all $j$, a contradiction.
\end{proof}

Finally, we treat the cyclic case.

\begin{lem}\label{lemma: 3}
For $\gamma \in \C$ and $d\geq 1$ a natural number, consider the (infinite) set of functions
$$f_0=\sqrt[d]{\frac{1}{x-\gamma}+1},f_1=\sqrt[d]{\frac{1}{x+1-\gamma}+1},f_2=\sqrt[d]{\frac{1}{x+2-\gamma}+1},\dots $$
Then for every $j\in \N$, $f_j$ has degree $d$ over $\C(x)(f_0,f_1,\dots,f_{j-1})$.
\end{lem} 
\begin{proof}
Define $E=\C(x)(f_0\dots,f_{j-1})$ and $L=\C(x)(f_j)$. We claim that $E$ and $L$ are linearly disjoint over $\C(x)$. Recall that $L/\C(x)$ is cyclic and in particular Galois of degree $d$. Hence it suffices to prove $E\cap L=\C(x)$. 

Since $f_j^d=(x+j+1-\gamma)/(x+j-\gamma)$, $L/\C(x)$ is totally ramified at $(x+j+1-\gamma)$ as $f_j$ has valuation $1/d$ at this place. Therefore, the subfield $E\cap L$ is also totally ramified at $(x+j+1-\gamma)$ of ramification index $$e=[E\cap L:\C(x)].$$ 

On the other hand, $E/\C(x)$ is unramified at $(x+j+1-\gamma)$. To see this, consider the chain of fields $E_0=\C(x)(f_0)$, $E_1=E_0(f_1),\dots,E_{j-1}=E_{j-2}(f_{j-1})=E$. Then it suffices to show that $E_m/E_{m-1}$ is unramified at all places $\mathfrak{p}$ lying over $(x+j+1-\gamma)$ for $m=1,\dots,j-1$. The minimal polynomial of $E_m/E_{m-1}$ divides $T^d-(x+m+1-\gamma)/(x+m-\gamma)$ which has coefficients in $\mathcal{O}_\mathfrak{p}$ and reduces to a separable polynomial modulo $\mathfrak{p}$. Hence $\mathfrak{p}$ splits completely in  $E_m/E_{m-1}$ (\cite[Thm. 3.3.7]{MR2464941}) and the claim follows. 

Therefore, the subfield $E\cap L$ of $E$ is also unramified at $(x+j+1-\gamma)$, so $e=1$ and we conclude $E\cap L=\C(x)$.
\end{proof}

\subsection{Constructing explicit $\s$-Picard-Vessiot rings}

In this section, we work over the base field $F=\C(x)$ considered as a $\ds$-field via $\de=\frac{d}{dx}$ and $\s(f(x))=f(x+1)$. For $m\in \N$, let $L(m)$ be as defined in Example \ref{ex direct limit}.(i).

\begin{prop}\label{prop: 123}
	Let $m\in \N$ and let $\H$ be either $\Gm$, $\Ga$ or a finite cyclic group. Then there exists a $\s$-Picard-Vessiot ring $R$ over $F=\C(x)$ with $\s$-Galois group isomorphic to $[\s]_\C\H$ such that $R\subseteq L(m)$ as an $F$-$\ds$-subalgebra.
\end{prop}
\begin{proof}
	We begin with $\H=\Gm$. Define $a=-(x-mi)^2 \in F$. Then $y=\exp(\frac{1}{x-mi})$ solves the differential equation $\de(y)=ay$ and $y$ is holomorphic on $\mathbb P^1_\C\backslash\{mi\}\supseteq V_{1,m}$. Hence $y \in F_{V_{1,m}}\subseteq L(m)$. Define $R=F\{y,y^{-1}\}\subseteq L(m)$. Then $R/F$ is a $\s$-Picard-Vessiot ring for $\de(y)=ay$ by Lemma \ref{lemma: crit}, since $L(m)^\de=\C=F^\de$. Using Proposition \ref{prop: Gm} together with Lemma \ref{lemma: 1}, we conclude that its $\s$-Galois group is isomorphic to $[\s]_\C\Gm$.
	
	\medskip
	
	We next treat the case $\H=\Ga$. Recall that the complex logarithm is a holomorphic function on $\C\backslash \R_-$, where $\R_-$ denotes the interval $(-\infty,0]$. As $x\mapsto \frac{1}{x-mi}+1$ defines a holomorphic function $V_{1,m}\to \C\backslash \R_-$, we conclude that $y=\log\left(\frac{1}{x-mi}+1 \right)$ is holomorphic on $V_{1,m}$ and in particular, $y\in F_{V_{1,m}}\subseteq L(m)$. Also note that $\de(y)$ is contained in $\C(x)=F$. Hence $R=F\{y\}\subseteq L(m)$ is a $\s$-Picard-Vessiot ring over $F$ with $\s$-Galois group isomorphic to $[\s]_\C\Ga$ by Proposition \ref{prop: Ga} together with Lemma \ref{lemma: 2}. 
	
	\medskip
	
	Finally we treat the case that $\H$ is a finite cyclic group. Let $d\in \N$ be the order of $\H$. 
	Recall that the complex $d$-th root $\sqrt[d]{x}=\exp(\log(x)/d)$ is a holomorphic function on $\C\backslash \R_-$, where $\R_-$ denotes the interval $(-\infty,0]$. As $x\mapsto \frac{1}{x-mi}+1$ defines a holomorphic function $V_{1,m}\to \C\backslash \R_-$, we conclude that $y=\sqrt[d]{\frac{1}{x-mi}+1}$ is holomorphic on $V_{1,m}$ and in particular, $y\in F_{V_{1,m}}\subseteq L(m)$. Hence $R=F\{y,y^{-1}\}\subseteq L(m)$ is a $\s$-Picard-Vessiot ring over $F$ with $\s$-Galois group $[\s]_\C\H$ by Proposition \ref{prop: cyclic} together with Lemma \ref{lemma: 3}. 
\end{proof}

\section{Main result}

The following lemma allows to base change a $\s$-Picard-Vessiot $R/F$ from $F$ to a $\ds$-overfield $F_2\supseteq F$ without shrinking the $\s$-Galois group under certain assumptions. An analogous statement for Picard-Vessiot rings has been proved in \cite[Lemma 2.9]{BachmayrHarbaterHartmannWibmer:DifferentialEmbeddingProblemsOverComplexFunctionFields}. 

\begin{lem}\label{lemma: compositum in diamond}
Let $(F,F_1,F_2,F_0)$ be a quadruple of $\ds$-fields with $F\subseteq F_1,F_2\subseteq F_0$ such that $F_1\cap F_2=F$ and $F_0^\de=F^\de$. Let $R/F$ be a $\s$-Picard-Vessiot ring such that $R\subseteq F_1$ as an $F$-$\ds$-subalgebra. Then the compositum $F_2R\subseteq F_0$ is a $\s$-Picard-Vessiot ring over $F_2$ with the same $\s$-Galois group as $R/F$. 
\end{lem}
\begin{proof}
Let $G$ be the $\s$-Galois group of $R/F$. By Lemma \ref{lemma: compositum} (ii), $F_2R/F_2$ is a $\s$-Picard-Vessiot ring with $\s$-Galois group $H$ a $\s$-closed subgroup of $G$. We consider the functorial invariants $\Frac(R)^H\subseteq \Frac(R)\subseteq F_1$. As $\Frac(R)^H\subseteq \Frac(F_2R)^H=F_2$ by Theorem \ref{theo: Galois correspondence}, we conclude $\Frac(R)^H\subseteq F_1\cap F_2=F$ and thus $H=G$ by Theorem \ref{theo: Galois correspondence}.  
\end{proof}

We are now in a position to prove our main theorem:
\begin{thm} \label{theo: main}
Consider $F=\C(x)$ as $\ds$-field with $\de=\frac{d}{dx}$ and $\s(f(x))=f(x+1)$. Then, for every linear algebraic group $\G$ over $\C$, there exists a $\s$-Picard-Vessiot ring over $F$ with $\s$-Galois group isomorphic to $[\s]_\C\G$. 
\end{thm}

\begin{proof}
Set $G=[\s]_\C\G$. By Proposition \ref{prop: generating}, there exist $\s$-closed subgroups $H_1,\dots,H_m$ of $G$ with $G=\langle H_1, \dots,H_m\rangle$ and such that for each $i$, $H_i\simeq [\s]_\C\H_i$, where $\H_i$ is $\Gm$, $\Ga$, or a finite cyclig group (depending on $i$). 

For $r\in \N$, let 
$$U_{r,m}=\{x \in \mathbb P^1_\C \mid x \notin \bigcup\limits_{l=1}^m\left( [-r,0]+li\right)\}$$ as in Example \ref{ex direct limit}.(ii)
and let $F_{U_{r,m}}$ denote the field of meromorophic functions on $U_{r,m}$. As explained in Example \ref{ex direct limit}, the direct limit $F(m)=\varinjlim\limits_{r\in \N} F_{U_{r,m}}$ is a $\ds$-field extension of $F$ with $F(m)^\de=\C$. We claim that there exists a $\s$-Picard-Vessiot ring $R(m)$ over $F$ with $\s$-Galois group $G$ and $R(m)\subseteq F(m)$ as a $F$-$\ds$-subalgebra. We prove the claim by induction on $m$.

If $m=1$, then $F(m)=F(1)=L(1)$ with $L(1)$ as defined in Example \ref{ex direct limit}.(ii) and the claim follows from Proposition \ref{prop: 123} applied to $m=1$. 

Now assume that we constructed a $\s$-Picard-Vessiot ring $R(m-1)\subseteq F(m-1)$ with $\s$-Galois group $\langle H_1,\dots, H_{m-1}\rangle$. We define connected metric open subsets $$V_{r,m}=\{x \in \mathbb P^1_\C \mid x \notin [-r,0]+mi\}$$ 
 as in Example \ref{ex direct limit}.(i). Hence $U_{r,m-1}\cap V_{r,m}=U_{r,m}$ is connected and $U_{r,m-1}\cup V_{r,m}=\X$ for all $r\in \N$. We define 
\begin{eqnarray*}
F_1&=&\varinjlim\limits_{r\in\N} F_{U_{r,m-1}}=F(m-1) \\
F_2&=&\varinjlim\limits_{r\in\N}  F_{V_{r,m}}=L(m) \\ 
F_0&=&\varinjlim\limits_{r\in\N}  F_{U_{r,m-1}\cap V_{r,m}}=F(m).
\end{eqnarray*}
As explained in Example \ref{ex diamond}, $(F,F_1,F_2,F_0)$ is a $\ds$-diamond  with the factorization property and moreover, $F_0^\de=\C=F^\de$ holds. By Lemma \ref{lemma: compositum in diamond}, we can thus lift $R(m-1)\subseteq F_1$ to a $\s$-Picard-Vessiot ring $R_2=F_2R(m-1)$ over $F_2$ with $\s$-Galois group $\langle H_1,\dots, H_{m-1}\rangle$ and $R_2\subseteq F_0$. 

On the other hand, Proposition \ref{prop: 123} provides us with a $\s$-Picard-Vessiot ring $\tilde R/F$ with $\s$-Galois group $H_m$ and $\tilde R\subseteq L(m)= F_2$ as a $F$-$\ds$-subalgebra. Again by Lemma \ref{lemma: compositum in diamond}, we may lift $\tilde R$ to a $\s$-Picard-Vessiot ring $R_1=F_1\tilde R$ over $F_1$ with $\s$-Galois group $H_m$ and $R_1\subseteq F_0$. 

Using the patching result Theorem \ref{thm: patching}, we obtain a $\s$-Picard-Vessiot ring $R(m)$ over $F$ with $\s$-Galois group $\langle H_1,\dots,H_m\rangle$ and with $R(m)\subseteq F_0=F(m)$ as claimed. 
\end{proof}

\section{Difference algebraic groups that are not $\s$-Galois groups}

Since, as shown above, all linear algebraic groups over $\C$, considered as $\s$-algebraic groups, occur as $\s$-Galois groups over $\C(x)$, one may feel tempted to conjecture that in fact all $\s$-algebraic groups over $\C$ occur as $\s$-Galois groups over $\C(x)$.
 However, this is not true. For example, as we will see below, the $\s$-closed subgroup $G$ of $\GL_1$ defined by $G(S)=\{g\in\GL_1(S)\ | \ g^2=1,\ \s(g)=g \}$ for any $\C$-$\s$-algebra $S$ is not a $\s$-Galois group over $\C(x)$. Moreover, we show that no proper non-trivial subgroup of the additive group $\mathbb{G}_a$ is a $\s$-Galois group over $\C(x)$.

 \subsection{A necessary criterion} 
In this section, we will isolate two properties that any $\s$-Galois group over $\C(x)$ must have. The main result of this section is the following.

\begin{thm} \label{theo: sgalois groups are sreduced and sconnected}
	If $G$ is a $\s$-Galois group over the $\ds$-field $\C(x)$ with derivation $\de=\frac{d}{dx}$ and endomorphism $\s$ given by $\s(f(x))=f(x+1)$, then $G$ is $\s$-reduced and $\s$-connected.
\end{thm}

The definition of $\s$-reduced and $\s$-connected is given below. Theorem \ref{theo: sgalois groups are sreduced and sconnected} is proved at the end of this section. The fact that every $\s$-Galois group over $\C(x)$ is $\s$-reduced follows rather directly from \cite{DiVizioHardouinWibmer:DifferenceGaloisTheoryOfLinearDifferentialEquations} and the fact that $\s\colon \C(x)\to \C(x)$ is surjective. On the other hand, the fact that every $\s$-Galois group over $\C(x)$ is $\s$-connected, essentially goes back to the fact that $\C(x)$ does not have any non-trivial finite difference field extensions.

We first discuss $\s$-connected $\s$-algebraic groups and the related notion of the group $\pis(G)$ of $\s$-connected components of a $\s$-algebraic group $G$. We do not strive for a comprehensive study of these notions. The interested reader is referred to Section 4.2 of \cite{Wibmer:Habil}. We will only introduce the definitions and results necessary for proving Theorem \ref{theo: sgalois groups are sreduced and sconnected}. Most of the required difference algebraic results appeared in \cite{TomasicWibmer:Babbit}. 

\medskip

Let $k$ be a $\s$-field. To motivate the definition of the group of $\s$-connected components of a $\s$-algebraic group, let us first recall the definition of the group of connected components of an algebraic group. See e.g., \cite[Chapter 6]{Waterhouse:IntroductiontoAffineGroupSchemes} or \cite[Section 2, g]{Milne:AlgebraicGroupsTheTheoryOfGroupSchemesOfFiniteTypeOverAField}.

Recall that a group scheme $\G$ of finite type over $k$ is \emph{\'{e}tale} if $k[\G]$ is an \'{e}tale $k$-algebra, i.e., $k[\G]\otimes_k\overline{k}$ is a a finite direct product of copies of $\overline{k}$.
The group $\pi_0(\G)$ of connected components of $\G$ can be defined through the following universal property:   
There exists a morphism $\G\to \pi_0(\G)$ of affine group schemes over $k$ such that $\pi_0(\G)$ is \'{e}tale and for every \'{e}tale group scheme $\H$ with a morphism $\G\to \H$, there exists a unique morphism $\pi_0(\G)\to \H$ such that
$$
\xymatrix{
	\G \ar[rr] \ar[rd] & & \pi_0(\G) \ar@{..>}[ld] \\
	& \H &	
}
$$
commutes. The identity component $\G^o$ of $\G$ can be defined as the kernel of  $\G\to \pi_0(\G)$.  The existence of $\pi_0(\G)$ can be established as follows: Let $\pi_0(k[\G])$ denote the union of the \'{e}tale $k$-subalgebras of $k[\G]$. Then one can show that $\pi_0(k[\G])$ is an \'{e}tale algebra and a Hopf subalgebra of $k[\G]$. So $\pi_0(\G)=\Spec(\pi_0(k[\G]))$ is an \'{e}tale algebraic group and the morphism $\G\to \pi_0(\G)$ corresponding to the inclusion $\pi_0(k[\G])\to k[\G]$ has the desired universal property.

To follow a similar path for difference algebraic groups we first need to define an appropriate difference analog of \'{e}tale algebras. Following \cite{TomasicWibmer:Babbit} we make the following definition.

\begin{Def} A \ks-algebra $R$ is \emph{$\s$-separable} if $\s\colon R\otimes_ kk'\to R\otimes_k k'$ is injective for every $\s$-field extension $k'$ of $k$. A $\s$-separable \ks-algebra is \emph{\ssetale} if it is an \'{e}tale $k$-algebra. 
\end{Def}

For a \ks-algebra $R$, we denote the union of all {\ssetale} \ks-subalgebras of $R$ with $\pis(R/k)$. If the base $\s$-field $k$ is clear form the context we will usually write $\pis(R)$ instead of $\pis(R/k)$. We know from \cite[Rem. 1.18]{TomasicWibmer:Babbit} that $\pis(R)$ is a \ks-subalgebra of $R$. 

\begin{Def}
	A $\s$-algebraic group $G$ is {\ssetale} if $k\{G\}$ is a {\ssetale} \ks-algebra.
\end{Def}

\begin{ex} \label{ex: ssetale}
	Let $G$ be the $\s$-closed subgroup of $\GL_1$ defined by $G(S)=\{g\in\GL_1(S)\ | \ g^2=1,\ \s(g)=g \}$ for any \ks-algebra $S$. Then $k\{G\}=k\times k$, with $\s$ given by $\s(a,b)=(\s(a),\s(b))$. Thus $G$ is {\ssetale}.
	
	On the other hand, the group $G$ defined by $G(S)=\{g\in\GL_1(S)\ | \ g^2=1,\ \s(g)=1 \}$ for any \ks-algebra $S$ is not {\ssetale}.	
\end{ex}

The following proposition was essentially proved in \cite{TomasicWibmer:Babbit}, but there the results were formulated in an algebraic manner. Here we give a more geometric interpretation.
\begin{prop} \label{prop: existense of pis}
	Let $G$ be a $\s$-algebraic group. Then there exists a {\ssetale} $\s$-algebraic group $\pis(G)$ together with a morphism $G\to \pis(G)$ of $\s$-algebraic groups satisfying the following universal property: If $H$ is a {\ssetale} $\s$-algebraic group with a morphism $G\to H$, then there exists a unique morphism $\pis(G)\to H$ such that 
	$$
	\xymatrix{
		G \ar[rr] \ar[rd] & & \pis(G) \ar@{..>}[ld] \\
		& H &	
	}
	$$
	commutes.
\end{prop}
\begin{proof}
	A \ks-Hopf algebra (\cite[Def. 2.2]{Wibmer:FinitenessProperties}) is a Hopf algebra over $k$ that has the structure of a \ks-algebra such that the Hopf algebra structure maps are morphisms of \ks-algebras. It follows directly from the definitions that the category of $\s$-algebraic groups is equivalent to the category of \ks-Hopf algebras that are finitely $\s$-generated as \ks-algebra (\cite[Rem.~ 2.3]{Wibmer:FinitenessProperties}).
	
	By Theorem 3.2 of \cite{TomasicWibmer:Babbit} the \ks-subalgebra $\pis(k\{G\})$ of $k\{G\}$ is a \ks-Hopf-subalgebra and by Theorem 4.5 of \cite{Wibmer:FinitenessProperties} every \ks-Hopf-subalgebra of a finitely $\s$-generated \ks-Hopf algebra is finitely $\s$-generated. It follows that the $\s$-algebraic group $\pis(G)$ represented by $\pis(k\{G\})$ is {\ssetale}. Moreover, since a quotient of a {\ssetale} \ks-algebra is {\ssetale} (\cite[Lemma 1.15]{TomasicWibmer:Babbit}) the inclusion $\pis(k\{G\})\subseteq k\{G\}$ has the following property: If $k\{H\}$ is a {\ssetale} \ks-Hopfalgebra with a morphism $\psi\colon k\{H\}\to k\{G\}$ then $\psi(k\{H\})\subseteq \pis(k\{G\})$. Geometrically, this translates to the required universal property. 
\end{proof}	

\begin{Def}
	Let $G$ be a $\s$-algebraic group. The {\ssetale} $\s$-algebraic group $\pis(G)$ from Proposition \ref{prop: existense of pis} is called the group of \emph{$\s$-connected components} of $G$. If $\pis(G)$ is the trivial group, then $G$ is \emph{$\s$-connected}.
\end{Def}

\begin{rem} \label{rem: sconnected}
	The terminology ``$\s$-connected'' is justified by the following fact (\cite[Lemma 4.2.35]{Wibmer:Habil}): A $\s$-algebraic group $G$ is $\s$-connected if and only if $\Spec(k\{G\})$ is connected with respect to the $\s$-topology. The closed sets of the $\s$-topology on $\Spec(k\{G\})$ are the invariant (Zariski) closed sets. 
\end{rem}

Example \ref{ex: ssetale} gives an example of $\s$-algebraic group that is not $\s$-connected. 
Any $\s$-algebraic group $G$ such that $k\{G\}$ is an integral domain is $\s$-connected because in that case $\Spec(k\{G\})$ is connected and so a fortiori $\s$-connected (Remark \ref{rem: sconnected}). One can also show that any (even non-connected) algebraic group is $\s$-connected when considered as a $\s$-algebraic group (\cite[Prop. 4.2.43]{Wibmer:Habil}).

\begin{ex}
	Let $G$ be the $\s$-closed subgroup of $\GL_2$ such that for a \ks-algebra $S$ and $g=\begin{pmatrix} a & b \\ c & d\end{pmatrix}\in \GL_2(S)$ we have $g\in G(S)$ if and only if
	\begin{align*}
	& c\s(a)+d\s(b)=0, \\ & a\s(c)+b\s(d)=0, \\ & a\s(a)+b\s(b)=c\s(c)+d\s(d), \\
	& (a\s(a)+b\s(b))^2=1, \\ & \s(a)\s^2(a)+\s(b)\s^2(b)=a\s(a)+b\s(b).
	\end{align*}
	To see that these equations indeed define a subgroup, note the following alternative description of $G$. Let $D$ be the $\s$-closed subgroup of $\GL_2$ defined by 
	$$D(S)=\left\{\begin{pmatrix} a & 0 \\ 0 & a\end{pmatrix}\in\GL_2(S) \ \Big| \ a^2=1, \ \s(a)=a\right\}$$
	for any \ks-algebra $S$. Then $G(S)=\{g\in\GL_2(S)\ | \ \s(g)g^T\in D(S)\}$. This is a subgroup of $\GL_2(S)$ because $D(S)$ lies in the center of $\GL_2(S)$.
	
	We claim that
	$\f\colon G\to D,\ g\mapsto \s(g)g^T$ satisfies the universal property of Proposition~\ref{prop: existense of pis}. In particular, $\pis(G)$ is isomorphic to the $\s$-algebraic group defined in Example \ref{ex: ssetale}.
	The easiest way to see this is to use some results from \cite{Wibmer:Habil}. In \cite[Prop. 4.2.41]{Wibmer:Habil} it is shown that if $\f\colon G\to D$ is a morphism of $\s$-algebraic groups such that 
	the dual map $\f^*\colon k\{D\}\to k\{G\}$ is injective, $D$ is {\ssetale} and $\ker(\f)$ is $\s$-connected, then $\f$ satisfies the universal property.
	
	We first note that $\f$ is a morphism of $\s$-algebraic groups because $D$ lies in the center of $\GL_2$ and that $D$ is {\ssetale} because it is isomorphic to the group in Example \ref{ex: ssetale}. To see that $\f^*$ is is injective, note that $\f^*$ sends the image of $x$ in $k[x]/(x^2-1)=k\times k=k\{D\}$ to the image of $a\s(a)+b\s(b)$ in $k\{G\}$. Thus if $\f^*$ was not injective, the latter image would need to equal $1$ or $-1$. However, neither is the case. For example, for the identity matrix $a\s(a)+b\s(b)$ evaluates to $1$, but for the matrix $\begin{pmatrix}(1,-1) & 0 \\ 0 & (1,-1) \end{pmatrix}\in G(S)$, where $S=k\times k$ with $\s(\lambda,\mu)=(\sigma(\mu),\sigma(\lambda))$, the expression $a\s(a)+b\s(b)$ evaluates to $-1$.
	Finally, the kernel $N$ of $\f$ is the $\s$-closed subgroup of $\GL_2$ defined by the equation $\s(g)g^T=I_2$. Since this equation can be rewritten as $\s(g)=(g^T)^{-1}$, we see that $k\{N\}$ is isomorphic to  $k[\GL_2]$. In particular, $k\{N\}$ is an integral domain. Thus $N$ is $\s$-connected.
\end{ex}

We will need two lemmas from \cite{TomasicWibmer:Babbit}.


\begin{lem}[{\cite[Lemma 1.25]{TomasicWibmer:Babbit}}] \label{lemma: pis and tensor product}
	Let $R$ and $S$ be \ks-algebras. Then $\pis(R\otimes_k S)=\pis(R)\otimes_k\pis(S)$.
\end{lem}

\begin{lem}[{\cite[Lemma 1.24]{TomasicWibmer:Babbit}}] \label{lemma: pis and base change}
	Let $R$ be a \ks-algebra and $k'/k$ an extension of $\s$-fields. Then $\pis(R\otimes_k k'/k')=\pis(R/k)\otimes_k k'$.	
\end{lem}	

We now return to the $\s$-Picard-Vessiot theory. As before, $F$ is a $\ds$-field of characteristic zero and $k$ is the $\s$-field of $\de$-constants of $F$.

%

\begin{prop} \label{prop: piss correspond}
	Let $R/F$ be a $\s$-Picard-Vessiot ring with $\s$-Galois group $G$. Then $\pis(R/F)$ is a $\s$-Picard-Vessiot ring over $F$ with $\s$-Galois group isomorphic to $\pis(G)$. 
\end{prop}
\begin{proof}	
	We will first show that $\pis(R/F)$ is a finite field extension of $F$. Note that if an \'{e}tale algebra (over a field) is an integral domain, then it is a field. Since $R$ is an integral domain and $\pis(R/F)$ is a union of \'{e}tale algebras contained in $R$ it follows that $\pis(R/F)$ is a field. So $\pis(R/F)$ is an algebraic field extension of $F$ and a $\s$-subfield of the field of fractions $E$ of $R$. Since $E$ is finitely $\s$-generated as a $\s$-field extension of $F$ and any intermediate $\s$-field of a finitely $\s$-generated $\s$-field extension is itself finitely $\s$-generated (\cite[Theorem 4.4.1]{Levin:difference}), it follows that $\pis(R/F)$ is finitely $\s$-generated over $F$. Since $\pis(R/F)$ is a union of finite $\s$-field extensions of $F$, we see that indeed $\pis(R/F)$ is a finite $\s$-field extension of $F$.
	
	To see that $\pis(R/F)$ is stable under the derivation $\de\colon R\to R$, let $a\in\pis(R/F)$ have minimal polynomial $f$ over $F$. Then $0=\de(f(a))=f^{\de}(a)+f'(a)\de(a)$, where $f^\de$ is the polynomial obtained from $f$ by applying $\de$ to the coefficients of $f$. Since $f'(a)\neq 0$, it follows that $\de(a)\in\pis(R/F)$. So $\pis(R/F)$ is an $F$-$\ds$-algebra. 
	
	We next apply $\pis(-/F)$ to the identity $R\otimes_F R=R\otimes_k k\{G\}=R\otimes_F(F\otimes_k k\{G\})$ from equation (\ref{eqn: alg torsor isom}). Using Lemmas \ref{lemma: pis and tensor product} and \ref{lemma: pis and base change} we find
	\begin{align*}\pis(R/F)\otimes_F\pis(R/F) & =\pis(R\otimes_F R/F)=\pis(R/F)\otimes_F\pis(F\otimes_k k\{G\}/F)=\\
	& =\pis(R/F)\otimes_F(F\otimes_k(\pis(k\{G\}/k))=\pis(R/F)\otimes_k\pis(k\{G\}/k)= \\
	&= \pis(R/F)\otimes_k k\{\pis(G)\}.
	\end{align*}
	Since $\pis(k\{G\}/k)\subseteq k\{G\}=(R\otimes_F R)^\de$, we see that
	$\pis(R/F)\otimes_F\pis(R/F)$ is generated by $(\pis(R/F)\otimes_F\pis(R/F))^\de$ as a left $\pis(R/F)$-module. Moreover, $\pis(R/F)^\de=F^\de$. Thus the extension $\pis(R/F)/F$ of differential fields is a Picard-Vessiot extension in the sense of \cite[Def. 1.8]{AmanoMasuokaTakeuchi:HopfPVtheory}. Since this definition is equivalent to the standard one (\cite[Theorem 3.11]{AmanoMasuokaTakeuchi:HopfPVtheory}) it follows that $\pis(R/F)$ is a Picard-Vessiot extension for some linear differential equation $\de(y)=Ay$, with $A\in F^{n\times n}$. Since $\pis(R/F)$ is an algebraic extension, the Picard-Vessiot ring and the Picard-Vessiot extension coincide. So the $F$-$\de$-algebra $\pis(R/F)$ is a Picard-Vessiot ring for $\de(y)=Ay$. It is then clear that the $F$-$\ds$-algebra $\pis(R/F)$ is a $\s$-Picard-Vessiot ring for $\de(y)=Ay$.  
	
	Since $(\pis(R/F)\otimes_F \pis(R/F))^\de=\pis(k\{G\}/k)$, we see that the $\s$-Galois group $\pis(R/F)$ is isomorphic to $\pis(G)$.	
	%
\end{proof}

\begin{cor} \label{cor: core trivial}
	Let $F$ be a $\ds$-field such that $F$ does not have any non-trivial finite difference field extensions. Then every $\s$-Galois group over $F$ is $\s$-connected.	
\end{cor}
\begin{proof}
	Let $R$ be a $\s$-Picard-Vessiot ring over $F$ with $\s$-Galois group $G$. As shown in the first paragraph of the proof of Proposition \ref{prop: piss correspond}, the $F$-$\s$-algebra $\pis(R/F)$ is a finite $\s$-field extension of $F$. Thus, by assumption, it is trivial and consequently also its $\s$-Galois group $\pi_0^\s(G)$ is trivial.
\end{proof}

Following \cite{DiVizioHardouinWibmer:DifferenceGaloisTheoryOfLinearDifferentialEquations}, a $\s$-algebraic group $G$ is \emph{$\s$-reduced} if $\s\colon k\{G\}\to k\{G\}$ is injective. 
The group in Example \ref{ex: ssetale} is $\s$-reduced. An example of a $\s$-algebraic group $G$ that is not $\s$-reduced would be $G=\{g\in \GL_n \ | \ \s(g)=I_n\}$.

\begin{proof}[Proof of Theorem \ref{theo: sgalois groups are sreduced and sconnected}] Let $R$ be a $\s$-Picard-Vessiot ring over $\C(x)$ with $\s$-Galois group $G$. In \cite[Cor. 4.4]{DiVizioHardouinWibmer:DifferenceGaloisTheoryOfLinearDifferentialEquations} it is shown that a $\s$-Galois group over a $\ds$-field $F$ such that $\s\colon F\to F$ is an automorphism is $\s$-reduced. Since $\s\colon \C(x)\to \C(x),\ f(x)\mapsto f(x+1)$ is an automorphism, it follows that $G$ is $\s$-reduced.
	
%
 The difference field $\C(x)$ (with $\s(f(x))=f(x+1)$) does not have any non-trivial finite difference field extensions (see the proof of Theorem XIX in Chapter 9 of \cite{Cohn:difference}). It thus follows from Corollary \ref{cor: core trivial} that $G$ is $\s$-connected.
\end{proof}

\subsection{Unipotent groups}
In this section, we show that being $\s$-reduced and $\s$-connected is far from sufficient for being a $\s$-Galois group over $\C(x)$. In fact, we show that no proper non-trivial $\s$-closed subgroup of $\Ga$ is a $\s$-Galois group over $\C(x)$. As a consequence of this result we also deduce that the constant subgroups of unipotent linear algebraic groups do not occur as $\s$-Galois groups over $\C(x)$.

\begin{lem}\label{lemma: additive}
Let $F$ be a $\ds$-field of characteristic zero with field of constants $k=F^\delta$. Let $G$ be a $\s$-closed subgroup of the additive group $\Ga$ over $k$ and let $R/F$ be a $\s$-Picard-Vessiot ring with $\s$-Galois group $G$. Then there exists an element $y \in R$ with $R=F\{y\}$ and $\delta(y)\in F$. 
\end{lem}
\begin{proof}
Let $X$ be the $\s$-variety over the $\s$-field $F$ represented by the $F$-$\s$-algebra $R$. As explained in the first paragraph of the proof of Prop \ref{prop: Tannaka}, there is a canonical structure on $X$ as a right $G_F$-torsor and we may thus also consider $X$ as a left $G_F$-torsor and use the results on left $G_F$-torsors in \cite{BachmayrWibmer:TorsorsForDifferenceAlgebraicGroups}. Either $G=\mathbb G_a$ and thus $X$ is the trivial $G$-torsor (\cite[Cor.\ 3.6]{BachmayrWibmer:TorsorsForDifferenceAlgebraicGroups}) or there exists an expression $\mathcal L(y)=\s^n(y)+\lambda_{n-1}\s^{n-1}(y)+\dots+\lambda_1\sigma(y)+\lambda_0y$ with $\lambda_i \in k$ such that $G(S)=\{g\in S\mid \mathcal L(g)=0\}$ for all \ks-algebras $S$ (\cite[Cor.\ A.3]{DiVizioHardouinWibmer:DifferenceAlgebraicRel}). In the latter case, it was shown in Example 5.4 of \cite{BachmayrWibmer:TorsorsForDifferenceAlgebraicGroups} that there exists an $a\in F$ such that $X$ is isomorphic as $G_F$-torsor to the torsor $X_a$ defined as the $\s$-closed $\s$-subvariety of $\mathbb A^1_k$ given by the equation $\mathcal L(y)=a$ with $G$-action given by addition. We conclude that in both cases, $X$ is a $\s$-closed  $\s$-subvariety of $\mathbb A^1_k$, thus there exists an $y\in R$ with $R=F\{X\}=F\{y\}$. Moreover, the $G_F$-action is given by addition, so
$$ g(\de(y))=\de(g(y))=\de(y+g)=\de(y)+\de(g)=\de(y)$$
for every $g\in G(S)$.
 It thus follows from the Galois correspondence that $\delta(y)\in F$.
\end{proof}

Recall that every element of $\C(x)$ has a unique partial fraction decomposition 
$$g(x)+\sum_{j=1}^{r} \frac{\alpha_j}{x+\beta_j}+\sum_{j=1}^{r_2} \frac{\alpha_{2j}}{(x+\beta_{2j})^2}+\dots+\sum_{j=1}^{r_m} \frac{\alpha_{mj}}{(x+\beta_{mj})^m} $$
with $g\in \C[x]$, $m,r,r_2,\dots,r_m \in \N$ and $\alpha_j,\alpha_{ij},\beta_j,\beta_{ij} \in \C$. The term $\sum_{j=1}^{r} \frac{\alpha_j}{x+\beta_j}$ is called the \textit{logarithmic part}. An element in $\C(x)$ has an antiderivative inside $\C(x)$ if and only if its logarithmic part vanishes.
\begin{prop}\label{prop do not occur}
 Let $G$ be a non-trivial proper $\s$-closed subgroup of the additive group $\Ga$ over $\C$. Then $G$ is not a $\s$-Galois group over the $\ds$-field $\C(x)$ with derivation $\de=\frac{d}{dx}$ and endomorphism $\s$ given by $\s(f(x))=f(x+1)$.
\end{prop}
\begin{proof}
 Suppose, for a contradiction, that $R/\C(x)$ is a $\s$-Picard-Vessiot ring with $\s$-Galois group $G$. By Lemma \ref{lemma: additive}, there exists an element $y \in R$ with $R=\C(x)\{y\}$ and $\delta(y)\in \C(x)$. Set $a=\delta(y)\in \C(x)$. As $\delta$ and $\sigma$ commute, $\delta(\s^l(y))=\s^l(a)$ for all $l\in \N$. By Proposition~\ref{prop: Ga}, there exists an $n\in \N$ such that $y,\s(y),\s^2(y),\dots,\s^n(y)$ are algebraically dependent over $\C(x)$. By the theorem of Kolchin-Ostrowski (see e.g. Section 2 in \cite{MR240106}), this can only happen if there exists a non-zero vector $(c_0,\dots,c_n) \in \C^{n+1}$ with $\sum_{l=0}^nc_l\s^l(y) \in \C(x)$. We may assume that $c_n\neq 0$. We differentiate and obtain that $\sum_{l=0}^nc_l\s^l(a)$ has no logarithmic part. Let $\sum_{j=1}^{r} \frac{\alpha_j}{x+\beta_j}$ be the logarithmic part of $a$ with pairwise distinct $\beta_1,\dots,\beta_r$ and non-zero elements $\alpha_j$. Then the logarithmic part of $\sum_{l=0}^nc_l\s^l(a)$ equals 
 $$0=\sum_{l=0}^n\sum_{j=1}^{r} \frac{c_l\alpha_j}{x+l+\beta_j}. $$ We claim that the logarithmic part of $a$ is zero, i.e., $\sum_{j=1}^{r} \frac{\alpha_j}{x+\beta_j}$ is an empty sum. Otherwise, we can choose $j_0$ with $1\leq j_0\leq r$ such that $\operatorname{Re}(\beta_{j_0})$ is maximal among all elements $\operatorname{Re}(\beta_j)$. Then the term with denominator $x+n+\beta_{j_0}$ cannot cancel with any other term (here we use that $\beta_1,\dots,\beta_r$ are pairwise distinct) and hence $c_n\alpha_{j_0}=0$, a contradiction. Therefore, the logarithmic part of $a$ vanishes and hence $y\in \C(x)$, $R=\C(x)$ and $G=\{1\}$, contradicting that $G$ is non-trivial.
\end{proof}

We consider expressions of the form $\mathcal L(y)=\s^n(y)+\lambda_{n-1}\s^{n-1}(y)+\dots+\lambda_1\sigma(y)+\lambda_0y$ with $\lambda_i \in \C$ and the corresponding $\s$-closed subgroups $G_{\mathcal L}$ of $\Ga$ with $G_{\mathcal L}(S)=\{g\in S\mid \mathcal L(g)=0\}$ for all $\C$-$\s$-algebras $S$. Every proper $\s$-closed subgroup of the additive group $\Ga$ is isomorphic to such a $G_{\mathcal L}$ (\cite[Cor.\ A.3]{DiVizioHardouinWibmer:DifferenceAlgebraicRel}).

\begin{rem}
 Note that $G_{\mathcal L}$ is $\s$-reduced if and only if $\lambda_0\neq 0$ and it is always $\s$-connected. Indeed, $\sigma\colon \C\{G_{\mathcal L}\} \to \C\{G_{\mathcal L}\}$ is injective if and only if $\lambda_0\neq 0$ and $\C\{G_{\mathcal L}\}$ is an integral domain, thus $G_\mathcal L$ is connected and in particular $\s$-connected. Hence most of the groups in Proposition \ref{prop do not occur} satisfy the necessary conditions given in Theorem \ref{theo: sgalois groups are sreduced and sconnected} but yet do not occur as $\s$-Galois groups.
\end{rem}

We remark that the multiplicative case differs from the additive case. Indeed, the constant points of the multiplicative group $\mathbb G_m$ do occur as $\s$-Galois group over $\C(x)$ (see Example~\ref{ex PVR}), whereas
Proposition \ref{prop do not occur} implies that the constant subgroup of the additive group does not occur as $\s$-Galois group over $\C(x)$.

To generalize this result from $\Ga$ to all unipotent groups, we need some basics about quotients of $\s$-algebraic groups (\cite[A.9]{DiVizioHardouinWibmer:DifferenceGaloisTheoryOfLinearDifferentialEquations} or \cite[Chapter 3]{Wibmer:Habil}): Let $G$ be a $\s$-algebraic group and $N$ a normal $\s$-closed subgroup. The quotient $G/N$ can be defined through the usual universal property. A morphism $\f\colon G\to H$ of $\s$-algebraic groups is a quotient map (i.e., can be identified with the canonical map $G\to G/N$ for some normal $\s$-closed subgroup $N$ of $G$) if and only if the dual map $\f^*\colon k\{H\}\to k\{G\}$ is injective.

\begin{cor}\label{cor: unipotent}
 Let $\G\leq\GL_n$ be a non-trivial unipotent linear algebraic group over $\C$ and let $G$ be the constant subgroup of $\G$, i.e., $G(S)=\{g\in \G(S) \mid \s(g)=g\}$ for all $\C$-$\s$-algebras $S$. Then $G$ is not a $\s$-Galois group over the $\ds$-field $\C(x)$ with derivation $\de=\frac{d}{dx}$ and endomorphism $\s$ given by $\s(f(x))=f(x+1)$.
\end{cor}
\begin{proof}
 Over a field of characteristic zero, every non-trivial unipotent linear algebraic group has a quotient isomorphic to $\mathbb G_a$ (\cite[Prop. 14.21 and Rem. 14.24 (a)] {Milne:AlgebraicGroupsTheTheoryOfGroupSchemesOfFiniteTypeOverAField}). Let $\pi\colon \G\to\Ga$ be such a quotient map. We claim that $\pi$ induces a quotient map $\f\colon G\to H$, where $H$ is the constant subgroup of $\Ga$, i.e., $H(S)=\{g\in S|\ \s(g)=g\}$ for all $\C$-$\s$-algebras $S$.
 
 The morphism $\pi\colon \G\to \Ga$ is given by a polynomial $p\in \C[X_{ij},1/\det(X)]=\C[\GL_n]$. So $\pi(g)=p(g)$ for $g\in\G(T)\leq \GL_n(T)$ and $T$ a $\C$-algebra.
 The induced morphism $G\to [\s]_k\Ga$ of $\s$-algebraic groups is also given by $g\mapsto p(g)$ for $g\in G(S)$ and $S$ a $\C$-$\s$-algebra and therefore maps into $H$. To see that the dual of $\f\colon G\to H$ is injective, note that the dual $\pi^*\colon \C[t]=\C[\Ga]\to \C[\G],\ t\mapsto \overline{p}$ of $\pi$ is injective because $\pi$ is dominant. The coordinate ring of $G$ is $\C\{G\}=\C[\G]$ with $\s$ the identity and similarly for $H$. Moreover, $\f^*$ agrees with $\pi^*$ as a morphism of $\C$-algebras. In particular, $\f^*$ is injective and so $\f$ is a quotient map.
 
 Now suppose, for a contradiction, that $G$ is a $\s$-Galois group over $\C(x)$. Then, by the second fundamental theorem of $\s$-Galois theory (\cite[Thm. 3.3]{DiVizioHardouinWibmer:DifferenceGaloisTheoryOfLinearDifferentialEquations}), also $H$ would be a $\s$-Galois group $\C(x)$. This contradicts Proposition \ref{prop do not occur}.
%
%
\end{proof}

\bibliographystyle{alpha}
 \bibliography{references}
\end{document}